\documentclass[12pt,reqno]{amsart}
\usepackage[latin1]{inputenc}
\usepackage{amssymb}
\usepackage{enumerate}
\usepackage[all]{xy}
\usepackage{color,graphicx}
\usepackage{hyperref}

\definecolor{darkgreen}{RGB}{65,165,118}
\definecolor{darkred}{RGB}{180,0,0}
\definecolor{darkblue}{RGB}{0,0,180}
\definecolor{ocre}{RGB}{204,119,34}
\definecolor{marro}{RGB}{103,64,58}

\newif\ifshowmynotecmd \showmynotecmdtrue

\DeclareMathAlphabet\EuR{U}{eur}{m}{n}
\SetMathAlphabet\EuR{bold}{U}{eur}{b}{n}
\newcommand{\gr}{\curs{Groups}}

\newcommand{\curs}{\EuR}

\newcommand{\Sets}{\curs{Sets}}

\newcommand{\sSets}{\mathbf{S}}

\newcommand{\Cdiag}{\mathbf{S}^{\calc}}
\newcommand{\aCdiag}{\mathbf{S}^{\acalc}}

\newcommand{\Gsets}{\textup{$\Gamma_\calc\op$-sets}}

\newcounter{let} 
\loop\stepcounter{let}
\expandafter\edef\csname cal\alph{let}\endcsname
{\noexpand\mathcal{\Alph{let}}}
\ifnum\thelet<26\repeat

\newcommand{\opdefnolim}[2][]{\expandafter\newcommand\csname#2\endcsname
{#1\operatorname{#2}\nolimits}}
\opdefnolim{Mor}  \opdefnolim{Hom}   \opdefnolim{Iso}   \opdefnolim{Aut}    \opdefnolim{Out}
\opdefnolim{Inn}    \opdefnolim{End}   \opdefnolim{Inj}    \opdefnolim{Rep}    \opdefnolim{Fib}
 \opdefnolim{Ker}    \opdefnolim{Coker}
 \opdefnolim{Tor}   \opdefnolim{Ext}
\opdefnolim{map}  \opdefnolim{aut}  \opdefnolim{imap} \opdefnolim{iaut}
\opdefnolim{hmap}  \opdefnolim{weq}\opdefnolim{hweq}
 \opdefnolim{Ob}
   \opdefnolim{haut}
  \opdefnolim{hofib}
  \opdefnolim{Gr}

\newcommand{\mapc}{\map_{\calc}}
\newcommand{\autc}{\aut_{\calc}}
 \newcommand{\opdef}[2][]{\expandafter\newcommand\csname#2\endcsname
{#1\operatorname{#2}}}
\opdef{colim}
\newcommand{\acalc}{\curs{a}\calc}

\renewcommand{\a}{\curs{a}}
\newcommand{\mapa}{\map_{\a}}

\makeatletter
\newcommand{\holim@}[2]{
  \vtop{\m@th\ialign{##\cr
    \hfil$#1\operator@font holim$\hfil\cr
    \noalign{\nointerlineskip\kern1.5\ex@}#2\cr
    \noalign{\nointerlineskip\kern-\ex@}\cr}}
}
\newcommand{\holim}{
  \mathop{\mathpalette\holim@{\leftarrowfill@\textstyle}}\nmlimits@
}
\makeatother

 \renewcommand{\Im}{\operatorname{Im}\nolimits}

\newcommand{\tdef}[2][]{\expandafter\newcommand\csname#2\endcsname
{#1\textup{#2}}}
\tdef{Id}    \tdef{incl}   \tdef{ct}   \tdef{pr}
  \tdef{proj}   \tdef{diag} \tdef{ev}
 \tdef[^]{op}

\newcommand{\N}{\mathbb{N}}

\newcommand{\defeq}{\overset{\text{\textup{def}}}{=}}

\newcommand{\conj}[2]{\{\,#1\,|\,#2\,\}}
\newcommand{\Conj}[2]{\bigl\{\,#1\,\bigm|\,#2\,\bigr\}}
\newcommand{\con}[1]{\{\,#1\,\}}

\newcommand{\widebar}[1]{\overset{\mskip2mu\hrulefill\mskip2mu}{#1}
                \vphantom{#1}}
\newcommand{\3}[1]{\widebar{#1}}
\newcommand{\4}[1]{\widehat{#1}}
\newcommand{\5}[1]{\widetilde{#1}}

\newcommand{\9}[1]{{}^{#1}\kern-1pt{}}

\newcommand{\longleft}[1]{\;{\leftarrow
\count255=0 \loop \mathrel{\mkern-6mu}
    \relbar\advance\count255 by1\ifnum\count255<#1\repeat}\;}
\newcommand{\longright}[1]{\;{\count255=0 \loop \relbar\mathrel{\mkern-6mu}
    \advance\count255 by1\ifnum\count255<#1\repeat\rightarrow}\;}
\newcommand{\Right}[2]{\overset{#2}{\longright#1}}
\newcommand{\RIGHT}[3]{\mathrel{\mathop{\kern0pt\longright#1}
        \limits^{#2}_{#3}}}
\newcommand{\Left}[2]{{\buildrel #2 \over {\longleft#1}}}
\newcommand{\LEFT}[3]{\mathrel{\mathop{\kern0pt\longleft#1}\limits^{#2}_{#3}}
}
\newcommand{\dRIGHT}[3]{\mathrel{
   \mathop{\vcenter{\baselineskip=0pt\hbox{$\kern0pt\longright#1$}
   \hbox{$\kern0pt\longright#1$}}}\limits^{#2}_{#3}}}
\newcommand{\LRIGHT}[3]{\mathrel{
   \mathop{\vcenter{\baselineskip=0pt\hbox{$\kern0pt\longleft#1$}
   \hbox{$\kern0pt\longright#1$}}}\limits^{#2}_{#3}}}
\newcommand{\RLEFT}[3]{\mathrel{
   \mathop{\vcenter{\baselineskip=0pt\hbox{$\kern0pt\longright#1$}
   \hbox{$\kern0pt\longleft#1$}}}\limits^{#2}_{#3}}}
\newcommand{\onto}[1]{\;{\count255=0 \loop \relbar\joinrel
    \advance\count255 by1
    \ifnum\count255<#1 \repeat \twoheadrightarrow}\;}

\newcommand{\dn}{{\downarrow}}

\newtheorem{Teo}{Theorem}[section]
\newtheorem{Pro}[Teo]{Proposition}
\newtheorem{Cor}[Teo]{Corollary}
\newtheorem{Lem}[Teo]{Lemma}

\newtheorem{Thm}{Theorem}

\theoremstyle{definition}
\newtheorem{Def}[Teo]{Definition}
\newtheorem{Nota}[Teo]{Remark}
\newtheorem{Ej}[Teo]{Example}
\newtheorem{Notation}[Teo]{Notation}

\theoremstyle{remark}

\title{Minimality in diagrams of simplicial sets}
\author{Carles Broto}
\address{Departament de Matem\`atiques, Universitat Aut\`onoma de
Barcelona, E--08193 Bellaterra, Spain}
\email{broto@mat.uab.cat}
\thanks{C.\ Broto and C.\ Giraldo are partially supported by FEDER-MINECO
Grant MTM2016-80439-P and AGAUR Grant 2017SGR1725}

\author{Ram\'on Flores}
\address{Departamento de Geometría y Topología, Universidad de Sevilla, E--41012, Sevilla, Spain}
\email{ramonjflores@us.es}
\thanks{R.\ Flores is partially supported by FEDER-MINECO Grant MTM2016-76453-C2-1-P}

\author{Carlos Giraldo}
\address{Facultad de Ciencias Naturales y Matem\'{a}ticas, Universidad del Rosario, 111711 Bogotá, Colombia}
\email{carlosan.giraldo@urosario.edu.co}
\thanks{}

\date{}

\begin{document}

\maketitle

\begin{abstract}

We formulate the concept of minimal fibration in the context of fibrations in the model category $\mathbf{S}^\mathcal{C}$ of
$\mathcal{C}$-diagrams of simplicial sets, for a small index category $\mathcal{C}$.
When $\mathcal{C}$ is an $EI$-category satisfying some mild finiteness restrictions,
we show that every fibration of $\mathcal{C}$-diagrams admits a well-behaved minimal model.
As a consequence, we establish a classification theorem for fibrations in
$\mathbf{S}^\mathcal{C}$ over a constant diagram, generalizing the classification theorem of
Barratt, Gugenheim, and Moore for simplicial fibrations \cite{Barrat}.

\end{abstract}

\section{Introduction}

After the introduction by D.~Kan in the fifties of the theory of simplicial sets and Kan fibrations, M.G.~Barratt,
V.K.~Gugenheim, and J.C.Moore developed in their celebrated paper \cite{Barrat} the appropriate version of simplicial fibre bundles, which in particular are Kan fibrations when the fibre is a Kan complex. Conversely, given a Kan fibration, these authors proposed the notion of a \emph{minimal simplicial fibration},
a fibrewise deformation retract of the original fibre map, which always exists and
 is itself a simplicial fibre bundle, provided the base of the fibration is connected.  In this way, the  classification of simplicial fibrations reduces to the classification of simplicial fibre bundles. The latter is achieved by associating to every fibre bundle a twisted cartesian product (TCP), where the twisting functions take values in the structure group $G$ of the bundle. Then, the simplicial group $G$ provides a
classifying space $\3{W}(G)$ and a universal fibre bundle such that every $G$-bundle over a base $B$ can be obtained as a pullback of a map $B\rightarrow\3{W}(G)$ along this universal bundle.

Denote by $\mathbf{S}$ the category of simplicial sets. Given a small category $\mathcal{C}$, the main goal of our work has been to extend Barratt-Gugenheim-Moore framework to the category $\mathbf{S}^{\mathcal{C}}$ of $\mathcal{C}$-diagrams of simplicial sets assuming that the fibrations involved are defined over a constant base. The first observation is that the category of diagrams possess a structure of cofibrantly generated model category and inherits a simplicial structure from $\mathbf{S}$ \cite{Hirschhorn}. Then, assuming always a constant base $B$, there exist natural translations of the notions of simplicial fibration, simplicial bundle or twisted cartesian product.

Extending the notion of minimal fibration is more involved. To this aim we define a full subcategory $\Gamma$
 of $\mathbf{S}^{\mathcal{C}}$ whose set of objects is given by the
 free $\mathcal{C}$-diagrams $\delta^c_n$ over the standard $n$-simplex $\Delta[n]$
 (see Definition \ref{D: Categ.Gamma}).
 Identifying  $\calc$-diagrams with $\Gamma\op$-sets,
 the free diagrams play in this context the role of the standard simplices in the classical theory, and become the basic building blocks
throughout we may study any $\mathcal{C}$-diagram $X$. In particular, a $n$-$c$-simplex will simply be a natural transformation from a free
diagram $\delta^c_n$ to $X$, for any $n\geq 0$ and $c\in\Ob(\calc)$,

Now it is possible to define a \emph{sub-$p$-homotopy}  relation for the set of
$n$-$c$-simplices of the total space of fibration of $\mathcal{C}$-diagrams for all  $n\geq 0$ and $c\in\mathcal{C}$
(Definition \ref{D.Homotop-C-simp.}),  and then the notion of
minimal fibration  (Definition \ref{D:Minimal fibration-diagrams}). If $\mathcal{C}$ is an
arbitrary small category the shape of the orbits of a free diagram can be elusive. However, when
 $\mathcal{C}$ is an \emph{artinian}
$EI$-category (Definition \ref{D:wee-descendant}), our main theorem establishes the existence of minimal $\calc$-fibrations:

\begin{Thm}\label{T:const-min-fibrat(A)}
Let $\mathcal{C}$ be a small artinian $EI$-category. If $p:X\Right1{} B$ is a fibration in $\mathbf{S}^{\mathcal{C}}$ for
which $X$ is a free diagram, then $p$ has a strong fibrewise deformation retract, $q:\hat{X}\Right1{} B$, which is  minimal and where $\hat{X}$ is a free $\calc$-diagram.
\end{Thm}

Observe that Example~\ref{ExampleNop} shows that some restrictions on the structure of the index
category are unavoidable in order to always find minimal models for the fibrations.

Now consider a simplicial set $B$, a category $\mathcal{C}$, and the composition $\mathcal{C}\rightarrow\{\ast\}\rightarrow \mathbf{S} $ of the trivial functor with the functor that takes the point to $B$. Slightly abusing notation, we will also denote this diagram by $B$ and call it a \emph{$\mathcal{C}$-diagram over a constant base $B$}.
Now, Theorem~\ref{T:const-min-fibrat(A)} opens the way to extend all the Barrat-Gugenheim-Moore program
to $\mathcal{C}$-diagrams over a constant base. As an intermediate step, we prove that that any $\mathcal{C}$-fibre
bundle over a constant diagram $B$ with fibre $F$ is equivalent to a $\mathcal{C}$-twisted cartesian product (and viceversa), and this allows to classify $\mathcal{C}$-fibre bundles. We denote by $aut_{\calc}(F)$ the simplicial group of automorphisms of the diagram $F$.

\begin{Thm}\label{T:Classif-C-fibre(B)} Let $\calc$ be a small category, $F$ a $\mathcal{C}$-diagram, $G$ a simplicial
subgroup of $aut_{\calc}(F)$ and $B$ a simplicial set. Then the homotopy classes of maps $[B,\overline{W}G]$ from $B$ to $\overline{W}G$ are in bijective correspondence with the equivalence classes of
$\mathcal{C}$-fibre bundles with fibre $F$, constant base $B$ and group $G$.
\end{Thm}

Now we can use Theorem \ref{T:const-min-fibrat(A)}, Theorem \ref{T:Classif-C-fibre(B)}, and Quillen's small object argument to classify
fibrations in $\mathbf{S}^\mathcal{C}$.

\begin{Thm}\label{T:Clas-C-Fib-1(C)} Let $\mathcal{C}$ be a small artinian $EI$-category,
$B$ a connected simplicial set, and $F$ a $\mathcal{C}$-diagram. Then there is a bijective
correspondence between the set $[B,B\haut_{\calc}(F)]$ of homotopy classes of maps
from $B$ to $B\haut_{\calc}(F)$ and the set of weak homotopy classes of fibrations
over the constant diagram $B$ whose fibres are weakly homotopy equivalent to $F$.
\end{Thm}

Roughly speaking, for a fibrant $\calc$-diagram $F$, $\haut_\calc(F)$ is described as
the set of self equivalences of $F$ in the homotopy category that admit a
rigidification in the category of simplicial sets. When $F$ is not fibrant we use
a fibrant replacement. The precise definition appears in Section~\ref{Classifibrations}, and we also
show there, elaborating on results of Dwyer-Kan \cite{Dwyer}, that it is a loop space.

This classification is related with the results of Dwyer-Kan-Smith in \cite{Dwyer2},
 classifying towers of fibrations. They construct inductively a classifying
 space $B(EY_1\ldots EY_n)$ by means of an appropriate homotopical version of
 the wreath product of groups. A tower of fibrations
$X_n\Right0{} X_{n-1}\Right0{}\dots \Right0{} X_1\Right0{} X_0$, where
$Y_i$ is the fibre of the map $X_i\Right0{} X_{i-1}$,
might be seen as a $\calc$-fibration over the constant base $X_0$, where
$\calc$ is the finite poset $\con{n\Right0{}\dots \Right0{}1}$, and with
fibre the diagram $F= F_n\Right0{} F_{n-1}\Right0{}\dots \Right0{} F_1$,
being $F_i$ the fibre of the composite fibration $X_i\Right0{} X_0$.
With this interpretation, the tower is classified by $B\haut_\calc(F)$, and this
turns out to be a connected component of Dwyer-Kan-Smith classifying space.

Part of the motivation for our study arose from our attempts to dualize to the augmented case classical results (\cite{Farjoun}, \cite{Bousfield}) of preservation of fibrations under localization functors. In this context, giving a fibration $F\rightarrow E\rightarrow B$ and a functor $L$, the problem is to produce another fibration $LF\rightarrow E'\rightarrow B$ that is naturally mapped from the original one. Our research about this subject, which corresponds to the case of a diagram of two objects and one morphism, will appear in a separate paper. Observe that the base of the fibration does not change throughout the process.

It is also worth mentioning the recent work of Blomgren-Chach\'olski \cite{Blomgren}.
Given objects $X$ and $F$, they define $\textrm{Fib}(X,F)$  as the category
whose objects are all the possible fibrations with base $X$ and fibre weakly homotopy
equivalent to $F$.  One of the innovative ideas of their work lies in the description
of the homotopy type of $\textrm{Fib}(X,F)$ and not only of its connected components,
thus providing a refinement of the classical classification theorems.
They develop the notion of \emph{core} of a category in order to overcome
set-theoretic difficulties that derive from the fact that $\textrm{Fib}(X,F)$ is not
a small category in general. Although their results are valid in any model category,
the point of view is different from ours, as we propose a more classical and
combinatorial approach  in the model category of $\mathcal{C}$-diagrams.

The paper is structured as follows. In Section \ref{S:Simpli-Prelim} we describe the structure of the category of diagrams of simplicial sets over a small index  category. The theory of minimal fibrations in this category is developed in Section \ref{MinFib}, where we prove Theorem \ref{T:const-min-fibrat(A)}. Section \ref{S:C-fibre bundles} is devoted to discuss and classify $\mathcal{C}$-fibre bundles, a goal that is achieved in Theorem \ref{T:Classif-C-fibre(B)}. Last section culminates the paper, by proving the classification result (Theorem \ref{T:Clas-C-Fib-1(C)}) for fibrations in the category of $\mathcal{C}$-diagrams. We conclude with an appendix, where we prove some technical properties of preordered sets that are needed in Section \ref{MinFib}..

\subsection*{Acknowledgments} The authors wish to thank Wojciech Chach\'olski for useful discussions about the paper. They also thank the Institute of Mathematics of the University of Seville and the Department of Mathematics of the Universitat Aut\`onoma de Barcelona for their kind hospitality hosting joint meetings of them.

\section{The model category of $\calc$-diagrams\label{S:Simpli-Prelim}}

The model category of $\calc$-diagrams has been extensively investigated in
\cite{Hirschhorn}.
In this section, we will first recall some basic facts about simplicial sets and will fix some notation (a more complete account of the theory can be found in \cite{Curtis, Goerss-Jardine, May}). Then, we
explain how the category of $\calc$-diagrams inherits the structure of cofibrantly
generated simplicial model category from that of simplicial sets. In the way, we
describe $\calc$-diagrams
as $\Gamma_\calc\op$-sets (see Definition~\ref{D: Categ.Gamma} below),
a point of view that will be useful in the construction of minimal fibrations in Section~\ref{MinFib}.

Let $\Delta$ be the category whose objects are the finite,  totally ordered, non-empty
sets $[n]=\{0,1,...,n\}$, $n\geq 0$, and whose morphisms are
the order-preserving functions. A simplicial set $X$ is a functor $\Delta\op\Right0{}\Sets$, while a simplicial map
is a natural transformation. We will denote by $\mathbf{S}$ the category of simplicial sets and
simplicial maps.

A simplicial set can be seen as a sequence of sets $X= \{X_n\}_{n\geq0} $,
together with structural maps
that consist of \emph{face operators}
$d_i\colon X_n \Right0{}X_{n-1}$, $n\geq1$, $0\leq i\leq n$, and of
\emph{degeneracy operators} $s_i\colon X_n\Right0{} X_{n+1}$, $n\geq0$, $0\leq i\leq n$,
subject to the following relations:
\begin{equation}\label{simprel}
\begin{split}
d_i d_j & = d_{j-1} d_i \,,\qquad  i < j \\
d_i s_j  & = s_{j-1} d_i \,,\qquad  i < j \\
d_j s_j &= 1 = d_{j +1} s_j  \\
d_i s_j &= s_j d_{i-1} \,,\qquad  i > j + 1 \\
s_i s_j &= s_{j+1} s_i \,,\qquad   i\leq j\,.
\end{split}
\end{equation}
Accordingly, a simplicial map $f\colon X\Right0{}Y$ between simplicial sets is a sequence of maps
$f_n\colon X_n\Right0{} Y_n$, for $n\geq 0$ which commutes with face and degeneracy operators.
The set of simplicial maps from $X$ to $Y$ will be denoted by $\Mor_\mathbf{S}(X,Y)$. When no
confusion is possible, a map between simplicial sets is understood to be a simplicial map.

The elements of $X_n$ are called \emph{$n$-simplices}, or just simplices if we do not need to emphasize the dimension $n$.
Simplices in the image of a degeneracy operator are called
\emph{degenerate} simplices. In turn, all the simplices of the form
$d_{i_1}d_{i_2}\dots d_{i_r}(x)$, $r\geq0$ will be called \emph{faces} of $x$. Those of the form $s_{i_1}s_{i_2}\dots s_{i_r}(x)$, $r\geq1$,
are called \emph{degeneracies} of $x$.

The \emph{standard} $n$-simplex $\Delta[n]$ is the simplicial set defined
$\Delta[n]=\Hom_\Delta([-],[n])$.
It contains a fundamental simplex $\imath_n = \Id_{[n]}\in \Delta[n]_n$ and all of simplices are either faces or
degeneracies of $\imath_n$. Actually, Yoneda's lemma provides
a natural isomorphism
$\Mor_{\sSets}(\Delta[n], X) \cong X_n$, assigning to a  simplicial map
$\varphi\colon\Delta[n]\Right0{} X$, the image of the fundamental simplex,
$\varphi(\iota_n)\in X_n$. We call $\varphi$ the \emph{characteristic map} of $x$.
We will often identify a simplex $x\in X_n$ with its
characteristic map $\Delta[n]\Right0{x}X$,  denoted with the same symbol.

Given a face $d_i\imath_n \in \Delta[n]$, the characteristic map is denoted
$d^i \colon \Delta[n-1]\Right0{}\Delta[n]$. Likewise, given a degeneracy $s_i(\imath_n)$, we
have a characteristic map $s^i\colon \Delta[n+1]\Right0{}\Delta[n]$. It turns out that the set of maps
$d^i$ and $s^i$, $0\leq n$, $0\leq i\leq n$, generate all maps $\Delta[m]\Right0{}\Delta[n]$, $n,m\geq0$,
subject to the dual relations \eqref{simprel}. In particular, we can identify $\Delta$ with the full subcategory
of $\sSets$ with objects $\Delta[n]$, $n\geq0$. Then, given an arbitrary simplicial set $X$,
the overcategory $\Delta\dn X$
might be seen as an index category for the simplices of $X$ and
we can express $X$  as colimit of its simplices (cf.~\cite[I.2.1]{Goerss-Jardine})
$$X\cong \colim_{\Delta\dn  X} \Delta[n].$$

We will now describe how these considerations extend to the category of $\calc$-diagrams,
for a fixed small index category $\calc$. Then, we will also explain how
the category of $\calc$-diagrams inherits a
cofibrantly generated simplicial model category structure from that of $\sSets$. We mainly follow
\cite{Hirschhorn} (in particular \cite[11.6.1]{Hirschhorn}) but at the same time we will introduce
the convenient notation for next sections.

Given a  small category $\mathcal{C}$, we denote by $\Cdiag$ the category of $\calc$-diagrams,
namely, the category of functors from $\calc$ to $\sSets$ and natural transformations between them.
Any simplicial set $X$ might be considered as a constant $\calc$-diagram, namely a diagram where any object
of $\calc$ maps to $X$ and any morphism maps to the identity of $X$.

Let $c$ be an object of $\calc$ and $\imath_c \colon \{c\}\to\calc$ the inclusion of the subcategory with a unique object $c$ and a unique morphism $id_c$. Given a
simplicial set $Y$, we will equally denote by $Y$ the functor defined on $\{c\}$ that assigns $Y$ to $c$. Then its \emph{left Kan extension} $\imath_{c,*}Y$ along the inclusion functor is a $\mathcal{C}$-diagram that can be described on objects as
$$
    \imath_{c,*}Y(d) = Y\times \Mor_\calc(c,d)
    \,,\qquad d\in \Ob(\calc)\,,
$$
and for a morphism $f\in\Mor_\calc(a,b)$, the induced map
$\imath_{c,*}Y(f)\colon \imath_{c,*}Y(a) \Right0{}\imath_{c,*}Y(b)$  is defined as
$\imath_{c,*}Y(f)(y,g) = (y,f\circ g)$,  for each simplex $y$ of $Y$ and $g\in\Mor_\calc(c,a)$.
(cf.\ \cite[11.5.25]{Hirschhorn}).
Furthermore, there is an adjunction
\begin{equation}\label{adjuntion1}
\Mor_{\Cdiag}(\imath_{c,*}Y,X)\cong  \Mor_{\sSets}(Y, X(c))\,.
\end{equation}
for each object $c$ of $\calc$.

Recall
that the \emph{boundary of $\Delta[n]$} is defined as the smallest simplicial
subset $\dot{\Delta}[n]$ of $\Delta[n]$ containing the
faces $d_i\imath_n$, $0\leq i\leq n$, while the $k^{th}$-\emph{horn} $\Lambda^k[n]$ ($0\leq k\leq n$,
$n\geq 1$) is the smallest simplicial subset $\Delta[n]$ containing all of the faces $d_i\imath_n$, $i\neq k$.
Now, we extend these definitions to the category of $\Cdiag$.

\begin{Def}\label{N:Deltas} Let $\calc$ be a small category. For each object $c$ of $\calc$ we
set
\begin{align*}
 \delta^c_n  & = \imath_{c,*}\Delta[n]\\
 \dot{\delta}^c_n &  = \imath_{c,*}\dot{\Delta}[n]\\
 \lambda^c_{n,k}  & = \imath_{c,*}\Lambda^{k}[n]
\end{align*}
for $n\geq 0$ and   $0\leq k\leq n$.
\end{Def}

\begin{Def}\label{D: Categ.Gamma} Let $\Gamma_\calc$ be the full subcategory of $\mathbf{S}^{\mathcal{C}}$ with objects $\delta_n^c$, with $c$ in $\calc$
 and all $n\geq 0$. We define the category $\Gamma_\calc\op$-sets as the category of functors from $\Gamma_\calc\op$ to the category of sets, and natural transformation between them.
\end{Def}

By naturality of the left Kan extension, we have morphisms in $\Gamma_\calc$,
\begin{gather*}
      s^i\colon \delta_{n+1}^c    \Right2{}   \delta_{n}^c \\
      d^i\colon \delta_{n-1}^c   \Right2{}   \delta_n^c
\end{gather*}
for all $c\in \Ob(\calc)$, $n\geq0$ and $0\leq i\leq n $, which are in turn induced by the simplicial maps  $s^i\colon \Delta[n+1] \Right0{}  \Delta[n]$ and
 $d^i\colon \Delta[n-1] \Right0{} \Delta[n]$, respectively. For every $f\in\Mor_\calc(a,b)$ and $n\geq0$ there is also a morphism
\begin{equation*}
\delta^f_n\colon \delta_{n}^b \Right2{} \delta_n^a
\end{equation*}
defined  by $(\delta^f_{n})_c(z,h)=(z,h\circ f)$, for any $c\in\Ob(\calc)$,  any simplex $z\in\Delta[n]$ and any morphism $h\in\Mor_\calc(b,c)$.

\begin{Pro} \label{structGamma} All morphisms of $\Gamma_\calc$ are  compositions of the form
$\delta^f_{m}d^{k_r}\dots d^{k_1}s^{j_1}\dots s^{j_t}$.

\end{Pro}

\begin{proof}
Take an arbitrary morphism
$\eta\colon \delta^c_n\Right0{}\delta^d_m$ of $\Gamma_\calc$.
The adjoint map is a map of simplicial sets
$$
\5\eta\colon \Delta[n]\Right0{} (\delta^d_m)_c = \Delta[m]\times \Hom_\calc(d,c) =
       \coprod_{\Hom_\calc(d,c)}\Delta[m]\,.
$$
Since the simplex $\Delta[n]$ is connected, the image $\5\eta(\Delta[n])$ is contained in one
of the components of $(\delta^d_m)_c $. Hence, we can factor $\5\eta$ as
a map $\varphi\colon \Delta[n]\Right0{}\Delta[m]$ followed by the inclusion
$\imath_f\colon \Delta[m] \Right0{} \Delta[m]\times \Hom_\calc(d,c) $ defined by $\imath_f(z)=(z,f)$.
It is known that every map $\Delta[n]\Right0{}\Delta[m]$ is a composition of face and degeneracy maps
$\varphi = d^{k_r}\dots d^{k_1}s^{j_1}\dots s^{j_t}$. So, $\5\eta$ is the composition shown in the left hand diagram
$$
\xymatrix{\Delta[n]\ar[d]_{d^{k_r}\dots d^{k_1}s^{j_1}\dots s^{j_t}}  \ar[rd]^{\5\eta} & \\
  \Delta[m]   \ar[r]^{\imath_f}   &  (\delta^d_m)_c
}  \qquad  \qquad
\xymatrix{\delta_n^c\ar[d]_{d^{k_r}\dots d^{k_1}s^{j_1}\dots s^{j_t}}  \ar[rd]^{\eta} & \\
  \delta_m^c   \ar[r]^{\delta_m^f }   &  \delta^d_m
}
$$
Taking adjoints we obtain the right hand diagram, and thus
$\eta=\delta_m^f d^{k_r}\dots d^{k_1}s^{j_1}\dots s^{j_t}$.
\end{proof}

According to Proposition~\ref{structGamma}, a $\Gamma_\calc\op$-set $X$ can be described as a
family of sets $\{ X_{c,n}\}$ indexed by objects $c$ of $\calc$ and
natural numbers $n\in \N$ together with structural maps
\begin{align*}
& d_i \colon X_{c,n} \Right2{}  X_{c,n-1}\,, & & 0\leq i\leq n\,, \ n\geq1\,, \  c\in \Ob(\calc)  \\
& s_i  \colon X_{c,n} \Right2{}  X_{c,n+1}    &   & 0\leq i\leq n\,, \ n\geq0\,, \   c\in \Ob(\calc)  \\
& f \colon X_{c,n} \Right2{}  X_{d,n}   & &  f\in \Mor_\calc(c,d)\,, \    n\geq0\,,
\end{align*}
satisfying the simplicial relations \eqref{simprel} and the naturality relations
\begin{equation}\label{Nat-cond}
    \text{$f\circ d_i =d_i\circ f$, and $f\circ s_i = s_i\circ f$, for all $f$ and  $i$.}
\end{equation}

\begin{Pro}
Given a small category $\calc$, there is a natural isomorphism of categories
$$
    \Gamma_\calc\op\textup{-sets} \Right1{\cong} \sSets^\calc\,.
$$
\end{Pro}
\begin{proof}
A $\Gamma_\calc\op$-set $X$ is mapped to a $\calc$-diagram $\5X$ where  $\5X(c)$ is the
simplicial set with
    $n$-simplices $X_{c,n}$,
     face maps  $d_i \colon X_{c,n} \Right0{}  X_{c,n-1}$, $1\leq n$, $0\leq i\leq n$
     and
     degeneracy maps $s_i  \colon X_{c,n} \Right0{}  X_{c,n+1}$, $0\leq n$, $0\leq i\leq n$.
We assign the map $\5X(f)\colon \5X(a)\Right0{} \5X(b) $
induced by the structural maps of $X$, $\delta_n^f$, $n\geq0$,  to every morphism
$f\in \Mor_\calc(a,b)$.

The inverse functor assigns to a $\calc$-diagram  $Y$ the $\Gamma_\calc\op$-set $\3Y$
given by $\3Y(\delta_n^c)= \Mor_{\sSets^\calc}(\delta_n^c, Y) \cong
            \Hom_{\sSets}(\Delta[n],Y(c)) \cong Y(c)_n$,
 with structural maps induced by those of each simplicial set $Y(c)$, $c\in \Ob(\calc)$, and
by $\delta_n^f\colon \delta_n^b\Right0{}\delta_n^a$, for each $n\geq 0$,  and each
$f\in\Mor(\calc)$.
\end{proof}

We will use this result as an identification of the categories
$\sSets^\calc$ and $\Gamma_\calc\op$-sets. In particular, by a $\calc$-diagram of simplicial sets we mean an object of
either of these categories, and we will choose the most useful description in each situation.

When a $\calc$-diagram $X$ is viewed as a $\Gamma_\calc\op$-set, every element of every set
$X_{c,n}= X(\delta_n^c)$ is called a simplex, and identifying
$X_{c,n}\cong \Mor_{\textup{$\Gamma_\calc\op$-sets}}(\delta_n^c, X)$, we equally
denote with the same letter a simplex $x\in X_{c,n}$ and its characteristic map
$\delta_n^c\Right0{x} X$, as we do for simplicial sets.
In this case we can also decompose a $\calc$-diagram $X$ as colimit of  its simplices
$$X\cong\colim_{\Gamma_\calc\dn  X} \delta_n^c\,.$$

\medskip

We will now explain the basic concepts and definitions that endow the category of $\calc$-diagrams with the structure of
a cofibrantly generated simplicial model category,
extending that of $\sSets$.

\begin{Def}
 A map of $\calc$-diagrams $p\colon X\Right0{} B$ is a \emph{fibration} if for every
diagram
\begin{equation*}
\xymatrix{ \lambda^c_{n,k}\ar[r]\ar@{^(->}[d]_i    &  X\ar[d]^p  \\
             \delta^c_n\ar[r]\ar@{.>}[ur]^{\theta}  &  B   })
\end{equation*}
where the solid arrows commute,
there is a map $\theta$, the dotted arrow,  making both triangles commutative.

A $\calc$-diagram $X$ is fibrant provided the unique map
$X\Right0{}*$ from $X$ to the constant diagram with value a point is a fibration.
\end{Def}
This reduces to the classical definition of Kan fibration when the index category $\calc$ is the trivial category with one object and one morphism.
It clearly holds that a map of diagrams $X\Right0{} B$ is a fibration if it restricts to a
Kan fibration $X(c)\Right0{}B(c)$ of every object $c\in \Ob(\calc)$.

\begin{Def}
We define \emph{weak equivalences of diagrams} as those maps $X\Right0{}Y$ that restrict to a
weak equivalence of simplicial sets $X(c)\Right0{}Y(c)$ for each  $c\in \Ob(\calc)$.
\end{Def}

\begin{Def}[{\cite[2.4]{Dwyer}}]\label{P:Basis} A map $f\colon X\Right0{} Y$ of $\mathbf{S}^{\calc}$ is called \emph{free}, if for every object $c$ of $\calc$ and $n\geq 0$, the map
$f_{c,n}\colon X_{c,n}\Right0{} Y_{c,n}$ is injective and if there exists a set
$\Sigma(f)$ of simplices of $Y$ such that
\begin{enumerate}[\rm (i)]
\item no simplex of $\Sigma(f)$ is in the image of $f$.
\item $\Sigma(f)$ is closed under degeneracy operators, and
\item for every object $c$ of $\calc$, $n\geq0$, and every simplex $y\in Y_{c,n}$ which
is not in the image of $f_{c,n}$, there exists a unique simplex $b\in\Sigma(f)$ and a unique map
$h\in\Mor(\calc)$ such that $\delta_n^h(b)=y$.
\end{enumerate}
We will say that a  $\calc$-diagram $X$ is \emph{free} if the map $\emptyset\Right0{} X$ from the empty diagram
is free.
\end{Def}

According to this definition a $\calc$-diagram $X$ is \emph{free} if there is a set $\Sigma=\Sigma(X)$
of simplices of $X$ which is closed under degeneracies and generates $X$ freely; namely, for each simplex $x\in X$, there is a unique $w\in \Sigma(X)$ and a unique morphism $h$ in $\calc$, such that
$\delta_n^h(w)=x$.
The set $\Sigma$  is called a \emph{basis} for $X$,
and an element of $\Sigma$ is called a \emph{generator}. We will call
 $\Sigma_n$ the subset of $\Sigma$ consisting of all $n$-simplices of $X$ that belong to
$\Sigma_n = \Sigma\cap X_n$.
Similarly, $\Sigma_c= \Sigma\cap X_c$
and $\Sigma_{c,n} = \Sigma\cap X_{c,n}$, for $n\geq0$ and $c\in\Ob(\calc)$.

In the case of a free map $f\colon X\Right0{}Y$, we will also call a basis for $f$ the set
$\Sigma(f)$ over which $f$ is free according to Definition \ref{P:Basis}.

\begin{Ej}\label{attach1} Let $X$ be a $\calc$-diagram. For any $n\geq 0$, $c\in \calc$,
and any map $\varphi$ in $\Mor_{\Cdiag}(\dot\delta_n^c,X)$, the map $\imath\colon X\Right0{} Y$
defined by the  push-out diagram
\begin{equation}\label{pushout}
  \xymatrix{\dot\delta_n^c  \ar[r]^{\incl}  \ar[d]_\theta &   \delta_n^c  \ar[d]^y \\
   X \ar[r]^{\imath} & Y }
\end{equation}
is a free map with
$\Sigma(\imath) =\{y\} \cup\{\text{iterated degeneracies of $y$}\}$,
where $y$ is the $n$-simplex of $Y$ classified by the map
$y\colon \delta_n^c\Right0{} Y$.
Moreover, if $X$ is free with basis $\Sigma(X)$, then $Y$
is also a free $\calc$-diagram with $\Sigma (Y) = \Sigma(X) \cup \Sigma(\imath)$
\end{Ej}

Actually, all free maps and free $\calc$-diagrams can be obtained by a possibly transfinite iteration of the above example.
A free map between $\calc$-diagrams, $\omega\colon X\Right0{}Y$ with base $\Sigma$, can easily be seen
to be the composition of a sequence
$$ X \Right2{} Y_0 \Right2{} Y_1 \Right2{} \dots \Right2{} Y_i \Right2{} Y_{i+1} \Right2{} \dots \Right2{}Y =\colim_iY_i
$$
where each step $Y_i\Right0{}Y_{i+1}$ is produced by a pushout diagram that generalizes
Example~\ref{attach1}
from attaching a single simplex to the case of a collection of simplices of dimension $i+1$.

Conversely, any transfinite iteration of the construction of Example~\ref{attach1} is a free map.
Actually, a standard transfinite induction argument shows that given an ordinal $\lambda$,
and a $\lambda$-sequence of maps of $\calc$-diagrams
$$
   X_0 \Right2{\varphi_1} X_1 \Right2{\varphi_2} X_2 \Right2{} \dots \Right2{\varphi_\beta}  X_\beta \Right2{} \dots
$$
with $\beta<\lambda$, where
 any of the maps $X_\beta \Right0{} X_{\beta+1} $, $\beta+1<\lambda$,
 is a free map  and
for any limit ordinal $\gamma<\lambda$, the induced map
$\colim_{\beta<\gamma} X_\beta \Right0{} X_\gamma$ is an isomorphism,
we can choose basis $\Sigma_\beta$ for the partial compositions $\omega_\beta\colon X_0 \Right0{} X_\beta$,
such that $\omega_{\alpha,\beta}(\Sigma_\alpha)\subseteq \Sigma_\beta$ if $\alpha<\beta$
and  the composition
 $$X_0\Right0{}\colim_{\beta<\lambda} X_\beta $$
is a free map with basis $\Sigma_\lambda =\colim_{\beta<\lambda}\Sigma_\beta$.

This shows in particular that in the category of $\calc$-diagrams the notion of
free maps and free $\calc$-diagrams of \cite{Dwyer}, as stated in
Definition~\ref{P:Basis}, coincides with that of  relative free cell complex
and free cell complex in \cite[11.5.35]{Hirschhorn}.
This justifies the next definition (cf.~\cite[2.4]{Dwyer}, \cite[11.6.1]{Hirschhorn}).

\begin{Def}
A map of $\calc$-diagrams is called a \emph{cofibration} if it is a retract of a free map.
Likewise, a $\calc$-diagram is \emph{cofibrant} if it is a retract of a free diagram.
\end{Def}

If the index category is the trivial category with one object and one morphism,
then the concept of cofibration reduces to the classical notion for simplicial sets,
where a map is a cofibration if it is injective.
If we have a category with two objects $a$ and $b$ and a unique non-identity morphism
$f\colon a\Right0{}b$, then a $\calc$-diagram  $X$ is free if and only if the map $X(f)$ is injective.

Given a discrete group $G$, a simplicial set $X$ with an action of $G$ can be seen as a
diagram defined over the category $\calb G$ with one object and the elements of $G$ as morphisms.
This is a free diagram if and only if $G$ acts freely on $X$.

\medskip

For an arbitrary small category $\calc$, if $Y$ is a simplicial set and $c$ is an object
of $\calc$, the $\calc$-diagram $\imath_{c,*}Y$ is a free diagram on  generators
 $\Sigma= Y\times\{\Id_c\}\subseteq  \imath_{c,*}Y(c) = Y\times \End_\calc(c)$.
 Likewise, if $f\colon A\Right0{}B$ is an inclusion of simplicial sets, then the
 induced map $\imath_{c,*}f\colon \imath_{c,*}A\Right0{}\imath_{c,*}B$ is a free map. This applies in the next definition:

\begin{Def}
Let $\calc$ be a small category. The \emph{generating cofibrations} of $\Cdiag$ are the maps
$$
    \dot{\delta}^c_n\Right2{}\delta^c_n\,,\qquad \text{for all $c\in\Ob(\calc)$, $n\geq 0,$}
$$
induced by the inclusions $\dot\Delta[n]\hookrightarrow \Delta[n] $.

The \emph{generating trivial cofibrations} are the maps
 $$\lambda^c_{n,k}\hookrightarrow \delta^c_n\,,\qquad \text{for all $c\in\Ob(\calc)$, $n>0$, and
$0\leq k\leq n$,}
$$
induced by the inclusions $\Lambda^k[n]\hookrightarrow \Delta[n]$.
\end{Def}

The category of $\calc$-diagrams also inherits a simplicial model category structure
from the category of simplicial sets, with structure natural functors, external product, function complex and exponent functor.
We will briefly recall these constructions.
The \emph{external product} with simplicial sets is defined
$$\times \colon \Gsets\times \sSets \Right2{} \Gsets$$
defined by
$(X\times K)_{c,n} =X_{c,n}\times K_n$, and with structural maps induced by those of $X$ and $K$.
Notice that we can similarly define a \emph{left product} with simplicial sets. We use both form without further explanation.

\begin{Def}
\label{functcomp}
The \emph{function complex}
$$\map_\calc \colon \Gsets\times  \Gsets\Right2{} \sSets\,,$$
is defined, for two $\calc$-diagrams $X$ and $Y$,
 as the simplicial set   $\map_\calc(X,Y)$ with
$n$-simplices $\Mor_{\Cdiag}(X \times \Delta[n], Y)$ or equivalently,
the commutative diagrams in $\Cdiag$
$$\xymatrix{ X\times \Delta[n]\ar[rr]^{\widetilde{\varepsilon}} \ar[dr]_{pr} &&
                 Y\times \Delta[n]\ar[dl]^{pr}  \\                                         & \Delta[n]   \rlap{ \,.} }$$
\end{Def}

Now, if $g\in\Mor_\mathbf{S}(\Delta[n]\times X, Y)$, its faces $d_ig$ and degeneracies $s_ig$ are given by the compositions
\begin{gather*}
\xymatrix{ \Delta[n-1]\times X\ar[r]^(0.55){d^i\times 1} & \Delta[n]\times X\ar[r]^(0.6)g  &  Y }
\qquad\text{and}\qquad
\xymatrix{ \Delta[n+1]\times X\ar[r]^(0.55){s^i\times 1} & \Delta[n]\times X\ar[r]^(0.6)g
&  Y. }
\end{gather*}
Finally, there is an \emph{exponent} functor
and
$$^\wedge \colon \sSets\times  \Gsets\Right2{} \sSets\,,$$
with $X^K= \map_\calc(K, X)$, where
$K$ is the constant $\calc$-diagram.

With the above definitions, the basic theory of simplicial sets extends to $\calc$-diagrams.
Indeed, $\Cdiag$ becomes a cofibrantly generated simplicial model category
\cite[11.6.1, 11.7.3]{Hirschhorn}.
In the next definition we collect the axioms for later reference.

\begin{Def}[{\cite[7.1.3, 9.1.6, and 11.1.2]{Hirschhorn}}]
\label{D:Model-Category}
Let $\mathcal{M}$ be a category equipped with three classes of morphisms called
\emph{fibrations}, \emph{cofibrations}, and \emph{weak equivalences}.
A fibration (respectively\ cofibration) which is also a weak equivalence is
called a \emph{trivial fibration} (respectively \emph{trivial cofibration}).
Then, $\mathcal{M}$ is a \emph{model category} \cite{Quillen} if it satisfies the following axioms:

\begin{enumerate}[\bf M1:]
\item The category $\mathcal{M}$ is closed under small limits and colimits.
\item Given composable maps $\xymatrix{X\ar[r]^g & Y\ar[r]^f & Z}$ in $\mathcal{M}$, if any two of $f$, $g$ and $f\circ g$ are weak equivalences, then so is the third.
\item  If $f$ is a retract of $g$ and $g$ is a weak equivalence, fibration or cofibration, then so is $f$.
\item Suppose that we are given a commutative solid arrow diagram
$$
 \xymatrix{ U \ar[r] \ar[d]_i& X\ar[d]^p \\
       V \ar@{.>}[ur] \ar[r]  &  Y
 }
 $$
where $i$ is a cofibration and $p$ is a fibration. If either $i$ is a trivial cofibration or $p$ is a
trivial fibration, then the dotted arrow exists and makes the diagram commutative.
\item Every map $f$ of $\mathcal{M}$ has two functorial factorizations:
\begin{enumerate}[\rm (a)]
\item  $f = p \circ  i$ where $p$ is a fibration and $i$ is a trivial cofibration, and
\item  $f = q \circ j$ where $q$ is a trivial fibration and $j$ is a cofibration.
\end{enumerate}
\end{enumerate}
A model category $\calm$ is said to be \emph{cofibrantly generated} if
\begin{enumerate}[\rm (1)]
  \item There exists a set $I$ of maps (called a set of \emph{generating cofibrations}) that permits the small object argument (\cite[10.5.15]{Hirschhorn}) and such that a
  map is a trivial fibration if and only if it has the $RLP$ with respect to every element of $I$.

\smallskip

  \item There exists a set $J$ of maps (called a set of \emph{trivial generating cofibrations})
that permits the small object argument and such that a map is a
fibration if and only if it has the $RLP$ with respect to every element of $J$.
\end{enumerate}

If $\calm$ is a model category, then it is a simplicial model category if it is enriched
over simplicial sets, and
the following two axioms hold:
\begin{enumerate}[\bf M1:]\setcounter{enumi}{5}
  \item For every two objects $X$ and $Y$ of $\mathcal{M}$ and every simplicial set $K$ there exist objects $X\otimes K$ and $Y^K$ of $\mathcal{M}$
       such that there are isomorphisms of simplicial sets
             $$\hom(X\otimes K,Y)\cong \hom(K,\hom(X,Y))\cong \hom(X,Y^K)$$
      that are natural in $X,Y$ and $K$.
  \item\label{axM7} If $i\colon A\Right0{} B$ is a cofibration in $\mathcal{M}$ and $p\colon X\Right0{} Y$ is a fibration in $\mathcal{M}$, then the map of simplicial
    sets
       $$(i^{*}, p_{*})\colon \hom(B,X)\Right0{} \hom(A,X)\times_{\hom(A,Y)}\hom(B,Y)$$
is a fibration that is a trivial fibration if either $i$
or $p$ is a weak equivalence.
\end{enumerate}
\end{Def}

\begin{Nota}\label{R:Quillen-SOA}
From the cofibrantly generated model structure for $\mathbf{S}^{\calc}$ and Quillen Small Object Argument (\cite{Quillen},\cite[10.5.16]{Hirschhorn}), we obtain that any map
$\emptyset\Right0{} X$ in $\mathbf{S}^{\calc}$ admits a functorial factorization
$\xymatrix{ \emptyset\ar@{^(->}[r] & QX\ar@{->>}[r]^{\sim} &  X }$, where $QX$ is free and the first map is a cofibration while the second map is a trivial fibration
\cite[10.5.2, 11.2.1-1]{Hirschhorn}.
\end{Nota}

The following consequence of the axioms will be useful in later sections.
The proof can be seen in \cite[9.3.8]{Hirschhorn}.

\begin{Pro}\label{cof} Let $\calm$ be a simplicial model category.
It $i \colon A\Right0{} B$ is a cofibration, and $K\subseteq L$ simplicial sets, the induced map
$$ A\otimes L \coprod_{A\otimes K} B\otimes K \Right2{} B\otimes L$$
is also a cofibration and it is a trivial cofibration if either  $i \colon A\Right0{} B$ is a trivial cofibration or
the inclusion $K\subseteq L$ is a weak equivalence.
\end{Pro}

\begin{Cor}\label{L:Lifting-Property} The map $\dot{\delta}^c_{n}\times\Delta[m]\cup\delta^c_{n}\times\Lambda^k[m]\hookrightarrow  \delta^c_{n}\times\Delta[m]$ is a
trivial cofibration.
\end{Cor}

\section{Minimal fibrations of diagrams of simplicial sets}\label{MinFib}

In this section we develop a theory of minimal fibrations in the context of diagrams of simplicial sets, that
generalizes the classical theory for simplicial sets
\cite{Barrat,Fritsch,Goerss-Jardine,Hovey,May}.
We will show that under some restrictions on the shape of the index
category $\calc$, the theory of minimal fibrations
carries out to the case of $\calc$-diagrams. We prove
Theorem~\ref{T:const-min-fibrat(A)} that gives sufficient conditions
under which a fibration of $\calc$-diagrams
is fibrewise equivalent to a minimal fibration.
Corollary~\ref{C:Isomorphism between minimal fibrations}  shows that two
fibrewise homotopic minimal fibrations are unique up to isomorphism, thus
the associated minimal fibration of Theorem~\ref{T:const-min-fibrat(A)}
is unique up to isomorphism. Recall that a fibre homotopy equivalence between
fibrations with the same base space $B$ is a map over $B$ that admits a
homotopy inverse, also over $B$.  It is understood that the homotopies
defining the homotopy inverse are  also maps over $B$.

The concept of minimal fibration in this context is defined below, in
Definition~\ref{D:Minimal fibration-diagrams}.
This is somewhat technical but it extends the classical definition and it seems
to be the appropriate choice to the arguments that follow. We will show that
under some restrictions on the shape of the index category every fibration with
free total $\calc$-diagram admits a strong fibrewise deformation retract which is a
minimal fibration and  unique up to isomorphism (Theorem~\ref{T:const-min-fibrat(A)} and
Corollary~\ref{C:Isomorphism between minimal fibrations}). Under the same restrictions,
Lemma~\ref{P:Other.Charact.Min.fibrat} gives an alternative characterization
of minimality that could be used as a more intrinsic definition of the same concept.

\begin{Def}
If $\calc$ is a small category and $p\colon X\Right0{} B$  a fibration of $\calc$-diagrams,
we say that two simplices $x,y\in X_{c,n}$ are \emph{$p$-homotopic},  $x\simeq_p y$,
 if there is a fibrewise homotopy from $x$ to $y$ relative to the boundary; that is,
 there is a homotopy
$$H\colon \delta^c_n\times\Delta[1]\Right2{}X,$$
 such that
 \begin{enumerate}[\rm (i)]
\item
$H_0\colon \delta^c_n \Right0{} X$ is the characteristic map of the simplex $x$ and
$H_1\colon \delta^c_n \Right0{} X$ is the characteristic map of  the simplex $y$,
\item $H|_{\dot \delta^c_n\times\Delta[1]}\colon \dot \delta^c_n\times\Delta[1] \Right0{}X$
is a constant homotopy,
\item  $p\circ H\colon \delta^c_n\times \Delta[1]\Right0{} B$ is a constant homotopy.
\end{enumerate}
\end{Def}

\begin{Def}\label{D.Homotop-C-simp.} Let $\calc$ be a small category and $p\colon X\Right0{} B$ a fibration of $\calc$-diagrams. We will say that a simplex
$y\in X$ is \emph{sub-$p$-homotopic} to $x$, or that $x$ \emph{precedes} $y$, $x \preceq y$, if there exists a morphisms $f\in\Mor(\calc)$ such that
$f(x)\simeq_p y$. We will write:
$$
[x,-) \defeq   \conj{y\in X}{x\preceq y} \,.
$$
\end{Def}

The following definition generalizes the concept of minimal fibration of simplicial sets
(\cite{Gugenheim}, \cite[10.1]{May}).

\begin{Def}\label{D:Minimal fibration-diagrams} Let $\calc$ be a small category.
A fibration $X\Right0{} B$ in $\mathbf{S}^{\mathcal{C}}$ is said to be
\emph{minimal} if $X$ is free and given any two generators $x,y$ of a base
$\Sigma$ of $X$, $x\preceq y$ implies $x=y$.
An object $X$ of $\mathbf{S}^{\mathcal{C}}$ is
called minimal if the map $X\Right0{}\ast$ is a minimal fibration.
\end{Def}

The next lemma shows that if the condition of this definition holds for
any one base then it holds for every other base.

\begin{Lem} Let $\calc$ be a small category and $p\colon X\Right0{} B$ a fibration of
$\calc$-diagrams where $X$ is a free $\calc$-diagram. If there is a base $\Sigma$
of $X$ such that given simplices $x,y\in\Sigma$, $x\preceq y$ implies $x=y$,
then the same holds for any other base of $X$.
\end{Lem}

\begin{proof}Assume that $\Sigma'$ is a different base set for $X$.
Then for any $x\in \Sigma'$ there is
a simplex $u\in\Sigma$ and an isomorphism $f\in\Iso(\calc)$  with $f(u)=x$.
Since $\Sigma$ is a base, given $x\in\Sigma$, there is a unique $u\in \Sigma'$ and a unique
morphism $f\in\Mor(\calc)$ such that $f(u)=x$. But $\Sigma'$ is a base, so there is $y\in \Sigma'$ and
$g\in\Mor(\calc)$ such that $g(y)=u$. Then $(f\circ g)(y)=x$, but both $y$ and $x$ are in $\Sigma'$,
and hence $x=y$ and $f\circ g$ is the identity morphism on the source of $g$. Now we also have
$(g\circ f)(u) = u $ and $u\in\Sigma$, so $g\circ f$ is also the identity over the source of $f$.
Thus $f$ is an isomorphism of $\calc$ with $f^{-1} = g$.

We can now conclude that $\Sigma'$ satisfies the  same condition as $\Sigma$.
Assume that  $x,y\in\Sigma'$ satisfy
that  $x\preceq y$, that is, there is $\varphi\in\Mor(\calc)$ such that $\varphi(x)\simeq_p y$.
We will conclude that $x=y$.
There are $u,v$ in $\Sigma$ and isomorphisms $f,g\in\Iso(\calc)$, such that
$f(u)=x$ and $g(v)=y$.  It follows that
$g^{-1}(\varphi(f(u))) = g^{-1}(\varphi(x)) \simeq_p  g^{-1}(y)=v$, so
$u\preceq v$ and, by the assumption on $\Sigma$,   $u=v$. Then, $(f\circ g^{-1})(y)=x$,
but both $x$ and $y$ are in the base set $\Sigma'$, so we have $x=y$.
\end{proof}

 Let $\calc$ be a small category. Given $a,b\in \Ob(\calc)$, we declare $a\sim b$ provided
 that there are
 morphisms $a\Right0{}b$ and $b\Right0{}a$ in $\calc$. This is an equivalence relation and
 the set of equivalence classes $\Ob(\calc)/{\sim}$ becomes a poset with the order relation
 defined by
 $[a]\leq [b]$ if and only if there is an arrow $a\Right0{} b$ in $\calc$. Notice that the
 definition
 of $\leq$ does not depend on the choice of representatives in the classes $[a]$ and $[b]$.

 \begin{Def} \label{D:wee-descendant} We will say that a small category $\calc$
 is \emph{artinian} if it satisfies the descending chain condition
on the poset $(\Ob(\mathcal{C})/{\sim},\leq)$. Namely, if every descending chain
$[a_1]\geq [a_{2}] \geq  [a_3]\geq \ldots $ in
             $\Ob(\mathcal{C})/{\sim}$  eventually stabilizes.
\end{Def}

 \begin{Def}\label{D:Subdiagram} Let $X$ be a $\mathcal{C}$-diagram in $\mathbf{S}$. A \emph{subdiagram} of $X$ consists of a subset $X'$ of simplices
 of $X$ which is in itself a $\calc$-diagram. This relation is denoted by
$X'\leq X$.
\end{Def}

The following is a classical notion (see for example \cite{Luck}):

\begin{Def}\label{D:EI-category} An $EI$-category is a  category $\mathcal{C}$ in which every endomorphism is an isomorphism.
\end{Def}

We need to introduce some notation now.

\begin{Notation}
Since $\Lambda^k[n]$ is the subcomplex of $\Delta[n]$ generated by all of the
faces $d_i\imath_n$ of $\Delta[n]$, except for the $k$th face $d_k\imath_n$,
a map $\alpha\colon \Lambda^k[n]\Right0{} X$ is thus determined by the restrictions to that faces.
Accordingly, we will denote
$$\alpha = (x_0,...,x_{k-1},-,x_{k+1},...,x_n)\,,$$
where $x_i=\alpha(d_i\imath_n)$ is the simplex of $X$ characterized by the restriction of $\alpha$ to the $i$th face. The analogue notation $\alpha =  (x_0,...,x_{k-1},x_{k},...,x_n)$
will be used for  maps $\alpha\colon \dot{\Delta}[n]\Right0{} X$, except that there are no blanks in this case.

More generally, a map $\varphi\colon Z\times \Lambda^k[n]\Right0{} X$ is written
$$\varphi = (\alpha_0,...,\alpha_{k-1},-,\alpha_{k+1},...,\alpha_n)\,,$$
where for each $i\neq k$, $\alpha_i$ is the composition
$Z \times \Delta[n-1]\Right2{\Id_Z\times d^i} Z\times \Lambda^k[n]\Right2{\varphi} X$.

These same conventions  extend to the case of $\calc$-diagrams. If $X$ is
a $\calc$-diagram, a map
$u\colon \dot\delta_c^n\Right0{}X$ is written $u=(u_0, \dots, u_n),$ where $u_i$ is the composition
$\delta_c^{n_1}\Right0{d^i}\dot\delta_c^n\Right0{u}X$.
\end{Notation}

Now we are ready to prove the main result of this section. It is stated as Theorem~\ref{T:const-min-fibrat(A)} in the introduction.

\begin{proof}[Proof of Theorem~\ref{T:const-min-fibrat(A)}]
Fix an artinian $EI$-category $\calc$,  a fibration $p\colon X\Right0{} B$
of $\calc$-diagrams, and assume that $X$ is a free $\calc$-diagram.
We will construct a strong fibrewise deformation retract $q\colon \hat{X}\Right0{} B$,
which is a minimal fibration of $\calc$-diagrams.

Consider a basis $\Sigma$ of $X$. Let $\Sigma'$ be a minimal subset of $\Sigma$ such
that:
\begin{enumerate}
\item\label{cond1} Every element of $\Sigma$ is preceded by an element
of $\Sigma'$, that is, for every $w \in\Sigma$ there exists some $x\in\Sigma'$ such that
$x\preceq w$.
\item\label{cond2} $\Sigma'$ contains degenerate representatives whenever it is possible.
That is, if $x\in\Sigma'$ and $[x,-) = [w,-)$ with $w$ a degenerate
simplex of $\Sigma$, then $x$ is a degenerate simplex.
\end{enumerate}

The set $\Sigma'$ exists, since $\mathcal{C}$ is a small artinian $EI$-category.
Proposition~\ref{E:Existen-of-A'} allows a choice of a minimal subset $\Sigma'$ satisfying
condition (1), and according to Remark~\ref{E:Existen-of-refined-A'} we can refine
the choice so that condition (2) is also satisfied.

\noindent \textbf{Step 1. } $\Sigma'$ is closed under degeneracy operators.

Choose an $n$-simplex $w\in \Sigma'$. Given a degeneracy operator $s_k$, $k=0,\dots, n$, we must show that $s_k(w)\in \Sigma'$.

Note that $\Sigma$ is closed under degeneracy operators, so $s_k w$ belongs to $\Sigma_{n+1}$.
Then, there exists a simplex $x\in\Sigma'_{n+1}$ such that $x\preceq s_k(w)$.
There also exists a map $f\in \Mor(\calc)$ such that $f(x)\simeq_p s_ k w$,
so in particular $d_kf(x) = d_ks_k(w)$, and then  $f(d_k x)=w$.
Since $w\in \Sigma'\subseteq\Sigma$ there exists
$g\in\Mor(\calc)$ such that $g(w)=d_k x$. By our choice of $f$ and $g$,
$gf$ is an endomorphism, and hence an automorphism, so let $(gf)^{-1}$ be its inverse map.
Now $x=(gf)^{-1}g f(x) \simeq_p (gf)^{-1} g (s_k w)$, so $s_k w \preceq x$. Thus,  $[x,-) = [s_k w,-)$ and then by condition (\ref{cond2}),
$x$ is a degenerate simplex.

Now, we have two degenerate simplices $x$ and $y=s_k(w)$ in $\Sigma$ and one precedes
the other. Since $x\preceq y$, there exists $f\in \calc$ such that
$f(x)\simeq_p y$. These are $p$-homotopic degenerate simplices in
a simplicial set, so $f(x)=y$ \cite[9.3]{May}; but $x$ and $y$ are in the same orbit,
and then they must coincide since both belong to $\Sigma$, that is, $x= y= s_kw \in \Sigma'$.

\noindent \textbf{Step 2.}
The set $\mathfrak{A}= \conj{(\4X,\4\Sigma)}{\text{$\4X\subseteq X$,
$\4X$ is free and $\4\Sigma$, and $\4\Sigma\subseteq \Sigma'$}}$
contains maximal elements for the order relation defined on $\mathfrak{A}$
by inclusion, namely $(\4X,\4\Sigma) \leq (\5X,\5\Sigma)$ if
and only if both $\4X\subseteq \5X$ and $\4\Sigma\subseteq \5\Sigma$.

The above relation clearly defines a partial order relation on $\mathfrak{A}$.
Moreover $\mathfrak{A}$ is non-empty, since it contains
the pair  $(\langle\Sigma'_0 \rangle, \4\Sigma'_0 )$, where $\langle\Sigma'_0 \rangle$
is the $\calc$-diagram  generated by all of the $0$-simplices of $\Sigma'$, and
$\4\Sigma'_0 $ consists of all of the elements of $\4\Sigma'_0$ and all their degeneracies.
More precisely, $\langle\Sigma'_0 \rangle$ is the $\calc$-subdiagram of $X$
whose set of $n$-simplices at any $c\in\calc$ is
$$
    \langle\Sigma'_0 \rangle_{c,n}=\conj{f(x)}{\text{$x\in s_0^n\Sigma_0'$ and
            $f$ is a morphism of $\calc$ with codomain $c$}},
$$
and $\4\Sigma'_0 = \coprod_{n\geq0} s_0^n(\Sigma_0')$. Now $\4\Sigma'_0 \subset \Sigma'$
since the latter is closed under degeneracy operators. It is also clear that
$\langle\Sigma'_0 \rangle$ is freely generated by $\4\Sigma'_0$.
Thus, $(\langle\Sigma'_0 \rangle, \4\Sigma'_0)\in \mathfrak{A}$.
If we take a chain $\{ (\4X_i, \4\Sigma_i)\}_{i\in I}$
in $\mathfrak{A}$, its union
$(\cup_{i\in I} \4X_i, \cup_{i\in I} \4\Sigma_i)$ is a $\calc$-subdiagram of
$X$ and it is free with base set $\underset{i\in I}{\cup}\4\Sigma_i\subseteq \Sigma'$.
Hence, it belongs to $\mathfrak{A}$ and it is an upper bound of the chain.
Now Zorn's Lemma implies that $\mathfrak{A}$ contains maximal elements.

\noindent \textbf{Step 3. }
Fix a maximal element $(\4X,\4\Sigma)$ of $\mathfrak{A}$.
If $w$ is an $n$-simplex of $X$ such that
\begin{enumerate}[\rm (i)]
\item $w$ is in the orbit of an element of $\Sigma'$, and
\item $d_k(w)\in \4X$ for each $k=0,\dots, n$,
\end{enumerate}
then  $w\in \4X$.

There is $x\in\Sigma'$ and $f\in\Mor(\calc)$ such that $f(x) = w$. Since $d_kw\in\4X$,
there exists $z\in\4\Sigma\subseteq\4X $ and $g\in\Mor(\calc)$ such that
$g(z)=d_k w$. In particular, $d_k x$, $d_k w$ and $z$ are in the same orbit of simplices of $X$.
Since $z\in\4\Sigma\subseteq \Sigma$, there is $h\in\Mor(\calc)$ such that $h(z)=d_k x$,
and thus $d_kx\in \4X$. As this is true for each $k=0,\dots,n$, it follows that
the subdiagram $\4X_1=\4X \cup\langle x\rangle$ of $X$  is free with base
$\4\Sigma_1=\4\Sigma\cup\{x\}\cup\{\text{iterated degeneracies of $x$}\}\subseteq\Sigma'$
(see \ref{attach1}).
Hence, $(\4X_1,\4\Sigma_1)\in \mathfrak{A}$ and $(\4X,\4\Sigma)\leq (\4X_1,\4\Sigma_1)$.
But $(\4X, \4\Sigma)$ is maximal in $(\mathfrak{A},\leq)$, and so
$\4X \cup\langle x\rangle=\4X $. Therefore $w\in\4X $, since $f(x)=w$.

\noindent \textbf{Step 4. }    Define $\mathfrak{B}$ as the set that consists of all
the pairs $(Y,H)$, where $Y$ is a subdiagram of $X$ containing $\4X $ and
$H\colon Y\times\Delta[1]\Right0{} X$ is a homotopy satisfying:

- $H_0$ maps $Y$ into $\4X$, being $H_0$ the composition
$$Y\cong Y\times\Delta[0] \Right2{1\times d^1} Y\times\Delta[1]\Right2{H} X\,.$$

- $H_1$ is the inclusion of $Y$ in $X$, being $H_1$ the composition
$$Y\cong Y\times\Delta[0] \Right2{1\times d^0} Y\times\Delta[1]\Right2{H} X.$$

- The restriction $H|_{\4X\times\Delta[1]}$ to $\4X\times\Delta[1]$, is constant, and then
equal to the composition     \linebreak
$\4X\times\Delta[1] \Right1{\pr} \4X\Right1{\incl} X.$

- $p\circ H$ is also constant, hence equal to the composition $Y\times\Delta[1] \Right1{\pr}Y\Right1{\incl} X\Right1{p}B$.

\smallskip

We can define a partial order `$\leq$' over $\mathfrak{B}$ as follows: $(Y,H)\leq (Y',H')$
if $Y\subseteq Y'$ and $H'|_{Y\times\Delta[1]}=H$. The diagram $\4X$
with the constant homotopy belongs to $\mathfrak{B}$, and so $\mathfrak{B}\neq\emptyset$.
Moreover, if we take a chain $\{(Y_i, H_i)\}_{i\in I }$ of $\mathfrak{B}$, its colimit $(Y_\infty, H_\infty)$, given by
$Y_\infty =\bigcup_{i\in I}Y_i$ and $H_\infty(x,t) = H_i(x,t)$ if $x\in Y_i$, is
an element of $\mathfrak{B}$, which is in turn an upper bound of the chain in $\mathfrak{B}$.
Then $\mathfrak{B}$ has maximal elements again by  Zorn's lemma.

\noindent \textbf{Step 5.}    If $(Y,H)$ is maximal in  $\mathfrak{B}$, then $Y=X$.

Assume that $Y\neq X$ and choose a simplex $z\in X$ of lowest dimension such that $z\notin Y$. It is non-degenerate (otherwise it would belong to $Y$) and its faces
belong to $Y$. Assuming
$z\in X_{d,n}$, let $z\colon \delta^d_n\Right2{} X$ be the map classified by $z$. Then, the restriction to the boundary of $\delta_n^d$ factors
through $Y$, so we denote by $\dot z\colon \dot\delta^d_n\Right2{} Y$ the restriction. Then, $Y\subseteq Y\cup_{\dot z}\delta_n^d \subseteq X$. We will extend the homotopy $H$ to
$Y\cup_{\dot z}\delta_n^d$.

To this aim, we first use the homotopy $H$ to find a new simplex, fibrewise
homotopic to $z$ (though not relative to boundary), with boundary inside $\4X$.
Since
                 \linebreak
$H\circ(\dot z\times 1)\colon \dot\delta^d_n\times \Delta[1]\Right0{}X$ coincides with
$z\colon  \delta^d_n\times \Delta[1]\Right0{}X$ in the intersection
$\dot\delta_n^d\times \{1\}$, we have a solid arrow commutative diagram
$$
\xymatrix{\dot\delta_n^d\times\Delta[1] \cup \delta_n^d\times\{1\}
        \ar[rr]^-{H\circ(\dot z\times 1)\cup z} \ar[d]_{\incl }& &  X \ar[d]^p \\
\delta_n^d\times\Delta[1]  \ar[rr]_{p\circ z \circ \textup{pr}} \ar@{.>}[rru]^G& & B.
}
$$
As the left vertical map is a trivial cofibration (see \ref{L:Lifting-Property}), according to the
homotopy lifting property there is a homotopy $G$ that makes the whole
diagram commutative.
Notice that $G_1=z$, while $G_0$ determines another simplex $z_1\in X$ with boundary
$\dot z_1\in \4X$. The restriction to the boundary is
$G|_{\dot\delta^d_n\times\Delta[1]}=H\circ (\dot z\times 1)$.

Now there exists a simplex $x\in \Sigma'$ that precedes $z_1$, $x\preceq z_1$, such that there exists
$f\in\Mor(\calc)$ with $y=f(x)\simeq_p z_1$. Since $\dot x = \dot z_1\in \4X$,
Step 3 implies that $y\in \4X$.

We now combine the homotopy $G$ with a chosen fibrewise homotopy $F$
between $z_1$ and $y$, relative to boundary.
This is achieved by means of the homotopy lifting property applied to the solid arrow commutative
diagram
$$
    \xymatrix{\delta^d_n\times \Lambda^0[2] \ar[d]_{\incl}  \bigcup
                     \dot\delta^d_n\times \Delta[2] \ar[d]_{\incl}
    \ar[rrrrr]^-{(-,G,F)\bigcup  H\circ(\dot z\times 1)\circ(1\times s^0)} & & & & & X \ar[d]^p  \\
           \delta^d_n\times \Delta[2] \ar@{.>}[rrrrru]^J \ar[rrrrr]^-{\ct}_-{= p\circ \pr}
                      && & & &  B \rlap{\,.}
}
$$
Let $J_0$ be the restriction of $J$ to the zero edge,  namely, the composition
$$\delta^d_n\times \Delta[1]  \Right1{1\times d^0}   \delta^d_n\times \Delta[2]\Right1{J}X.$$
Then, $J_0$ is a fibrewise homotopy from $y$ to $z$, with restriction to the boundary
$J_0|_{\dot\delta^d_n\times \Delta[1]}= H\circ(\dot z\times1)$. This equality makes possible an extension $\5H\colon   (Y\cup_{\dot z}\delta_n^d)\times \Delta[1] \Right0{} X$ of $H$,
that is defined by the push-out diagram:
$$
\xymatrix{  \dot\delta^d_n\times\Delta[1] \ar[r]^{\dot z\times 1}  \ar[d]_{\incl\times1}
                                &    Y\times\Delta[1]  \ar[d]  \ar@/^3ex/[rrdd]^H    &  & \\
\delta^d_n\times\Delta[1] \ar[r]^{}   \ar@/^-3ex/[rrrd]_{J_0}
                                &   (Y\cup_{\dot z}\delta_n^d)\times \Delta[1]   \ar@{.>}[rrd]^{\5H}& \\
 & & & X\rlap{\,.}
}
$$
After identifying $Y\cup_{\dot z}\delta_n^d$ with its image
$Y\cup\langle z\rangle= \Im(\incl\cup_{\dot z} z)\subseteq X$, the pair
$(Y\cup\langle z\rangle, \5H)$ contradicts
the maximality of $(Y,H)$. It follows that  $Y=X$.

\noindent \textbf{Step 6. }   The restriction $q=p|_{\4X }$ is a minimal fibration.

Since $p|_{\4X }$ is a retraction of $p$, it is a fibration. By construction $\4X $ is a free diagram and if $w$ is a generator of $\4X $, then it is also
a generator of $X$ that belongs to $\Sigma'$. Then, none of such simplices can be sub-$p$-homotopic, as otherwise the minimality condition that satisfies $\Sigma'$ would be contradicted. Therefore $q$ is a minimal fibration.
\end{proof}

Now we will establish a minimality condition for fibrations.

\begin{Pro}\label{P:Other.Charact.Min.fibrat} Let $\mathcal{C}$ be an artinian $EI$-category. Assume that $X$ is a free $\calc$-diagram
and $p\colon X\Right0{} B$ is a fibration in $\mathbf{S}^{\mathcal{C}}$. Then $p$ is minimal if and only if any strong fibrewise
deformation retract of $X$ coincides with $X$.
\end{Pro}

\begin{proof}
 Assume that $p\colon X\Right0{} B$ is minimal and let $D\leq X$ be a strong fibrewise deformation retract of $X$, and denote by $j$ and $r$ the inclusion and the
 retraction, respectively. There is a commutative diagram of fibrations:
 $$
 \xymatrix{ D \ar[r]^j \ar[rd]_{p|_D} & X \ar[d]^p \ar[r]^r & D\ar[ld]^{p|_D} \\
   &   B  &
 }
 $$
Assume that $D$ and $X$ coincide up to dimension $n-1$, and let $z$ be a simplex of $X$ of dimension $n$,
with characteristic map $\delta_n^c\Right1{z}X$.
Assume furthermore that $z\in \Sigma = \Sigma(X)$ belongs to the generating set of $X$.

Let $H\colon X\times \Delta[1]\Right0{} X $ be a fibrewise homotopy relative to $D$ between the identity $\Id_X$ and $r\circ j$. Since the boundary $\dot z$ of the
simplex $z$ belongs to $D$, the composition
$$\delta_n^c\times \Delta[1]\Right2{z\times1}X\times\Delta[1] \Right2{H} X$$

\noindent defines a $p$-homotopy $z\simeq_p jr(z)$. If $w$ is the generator of the orbit of $jr(z)$, $w\preceq z$. Since both $z$ and $w$ are in $\Sigma$ and
$p\colon X\Right0{} B$ is minimal, they must coincide. Hence there is an endomorphism $f$ of $\calc$ such that with $f(z)=jr(z)$; but as $\calc$ is an
$EI$-category, $f$ is an isomorphism, and thus $z= f^{-1}(r(z))$. Since $r(z)$ belongs to $D$, it holds that $z\in D$. So we have proved that every generator of $X$ in dimension $n$ belongs to $D$; and as they coincide in dimension $n-1$, we obtain by induction the desired equality $D=X$.

The other implication is a consequence of Theorem \ref{T:const-min-fibrat(A)}.
\end{proof}
The next proposition states a useful technical property of minimal fibrations.

\begin{Pro}\label{P:strong-isom-min.f} Let $\mathcal{C}$ be a small $EI$-category, $p\colon X\Right0{} B$ a fibration of $\calc$-diagrams and $Z$ another $\calc$-diagram. Assume that $\alpha,\beta\colon Z\Right0{} X$ are fibrewise homotopic maps and $p$ is a minimal fibration. If $\beta$  is an isomorphism, then $\alpha$ is also an
isomorphism.
\end{Pro}

\begin{proof} We will show that $\alpha$ is an isomorphism by induction on the dimension $n$. Notice that for negative dimension $\alpha$ is the isomorphism between
empty sets. Hence, we will assume that $\alpha$ is an isomorphism in dimension $k<n$, and will prove that it is also an isomorphism in dimension $n$.

\noindent\textbf{Step 1.} The map $\alpha$ is surjective in dimension $n$.

It is sufficient to show that every $n$-simplex of $\Sigma =\Sigma (X)$ is in the image of $\alpha$. Let $x$ be an $n$-simplex of $\Sigma=\Sigma(X)$. By induction hypothesis, $\dot x\in\Im(\alpha)$, and then if $x\colon \delta_n^c\Right0{}X$ is the characteristic map
for $x$ (that we still denote by the same symbol), there is a map $u\colon \dot\delta_n^c\Right0{} Z$ and a commutative diagram in $\Cdiag$:
$$
\xymatrix{ \dot\delta_n^c   \ar[d]_{\incl} \ar[rr]^{(u_0,...,u_n)}    &&     Z  \ar[d]^{\alpha}    \\
\delta_n^c   \ar[rr]^x   &&      X.
}
$$

Let $H\colon Z\times \Delta[1] \Right0{} X$ be a fibrewise homotopy from $\beta$ to $\alpha$, that is, $H_0=\beta$ and $H_1=\alpha$. Then, the above diagram extends
to a new diagram of solid arrows
$$
\xymatrix{ \dot\delta_n^c  \times \Delta[1] \cup  \delta_n^c\times \{1\}
           \ar[d]_{\incl} \ar[rrr]^-{H\circ((u_0,...,u_n)\times1)\cup x }    & & &   X  \ar[d]^p   \\
\delta_n^c  \times \Delta[1]\ar@{.>}[rrru]^G  \ar[rrr]_-{\ct_{p(x)}} & & &   B,
}
$$ where $\ct_{p(x)}$ is the constant homotopy of the characteristic map of the simplex  $p(x)\in B$: $\delta_n^c\Right1{x} X\Right1{p} B$.  Corollary~\ref{L:Lifting-Property} implies that the left side vertical map is a trivial cofibration, and so Axiom M4 of \ref{D:Model-Category} implies that there is a homotopy
$G\colon  \delta_n^c  \times \Delta[1] \Right0{} X$ making the whole diagram commutative. By construction $G_1$ coincides with the characteristic map of $x$. Then,
$w=G_0$ determines a new simplex $w\in X$ such that $p(w)=p(x)$ and $w$ and $x$ are homotopic, although not relative to the boundary. In fact,
$G|_{\dot\delta_n^c  \times \Delta[1] } = H\circ ((u_0,...,u_n)\times1)$ is not necessarily a constant homotopy.

Notice that by construction $\dot w = (\beta(u_0),...,\beta(u_n))$, and since $\beta$ is an isomorphism, there is a simplex $z\in Z$ such that $\beta(z)= w$, and
then $\dot z = (u_0,...,u_n)$. Now, we want to compose suitably the homotopies  $G$ from $w$ to $x$ and $H\circ (z\times 1)$, which is a homotopy from $\beta(z)=w$
to $\alpha(z)$. This is achieved by obtaining again a lifting $F$:
$$\xymatrix{   & \dot\delta_n^c  \times \Delta[2] \cup  \delta_n^c\times \Lambda^0[2]\ar[d]_{\incl} \ar[rrrrr]^-{H\circ((u_0,...,u_n)\times s^1)\cup (-,H\circ(z\times 1),G) }    &  & && &  X  \ar[d]^p   \\
\delta_n^c  \times \Delta[1] \ar[r]^-{1\times d^0} &  \delta_n^c  \times \Delta[2]\ar@{.>}[rrrrru]^F  \ar[rrrrr]_-{\ct_{p(x)}} & & & & &   B \rlap{\,.} }
$$

The composite $F\circ (1\times d^0)$ is now a homotopy from $x$ to $\alpha(z)$. The restriction to the boundary is
$F\circ (1\times d^1)|_{\dot\delta_n^c\times\Delta[1]}=H\circ (u\times s^1) \circ (1\times d^0)  = H\circ (u\times d^0s^0)$, which is the constant homotopy of
$H_1\circ u$, the characteristic map of $\alpha(u)$, and then $\dot x = \alpha(\dot z) = \alpha(u)$. Let $y\in\Sigma$ be the generator of the orbit of the simplex
$\alpha(z)$ in $X$. We have shown that $y\preceq x$, and hence $y=x$ since both are in $\Sigma$. Since $\calc$ is an $EI$-category, there is an isomorphism $f\in\calc$ such that $f(x)=\alpha(z)$. Hence, $x= \alpha(f^{-1}z) \in \Im (\alpha)$.

\noindent\textbf{Step 2. }
If $z$ and $w$ are two $n$-simplices of $Z$ such that $\alpha(z)=\alpha(w)$, then $\beta(z)\simeq_p \beta(w)$.

By the induction hypothesis, $z$ and $w$ must have the same boundary. Since $\alpha(z)=\alpha(w)$, also $\alpha(d_iz)=\alpha(d_iw)$ for all $i=0,\dots,n$. But
$\alpha$ is injective in dimension $n-1$, and thus $d_iz=d_iw$ for all $i=0,\dots,n$, that is, $\dot z = \dot w$.

As in Step 1, let $H$ be a fibrewise homotopy from $\beta$ to $\alpha$. Then, $H\circ(w\times 1)$ is a homotopy from $\beta(w)$ to $\alpha(w)$ (or rather the
corresponding characteristic maps). Similarly, $H\circ(z\times 1)$ is a homotopy from $\beta(z)$ to $\alpha(z)$. But $\alpha(w)=\alpha(z)$, so we will combine both homotopies in order to get another one between $\beta(w)$ and $\beta(z)$.

Notice that $H\circ(w\times 1)$ and $H\circ(z\times1)$ coincide over the boundary $\dot w= \dot z$, and denote by
$$h = H\circ(w\times 1)|_{\dot\delta_n^c \times \Delta[1]} =  H\circ(z\times1)|_{\dot\delta_n^c \times \Delta[1]} $$

the restriction. We have a commutative diagram of solid arrows
$$\xymatrix{ & \delta_n^c \times \Lambda^2[2] \cup \dot\delta_n^c \times \Delta[2]   \ar[d]_\incl \ar[rrrrr]^-{ ( H\circ(z\times1), H\circ( w\times1),-   ) \bigcup  h\circ (1\times s^0) } &&&&& X \ar[d]^p\\
  \delta_n^c \times \Delta[1] \ar[r]^{(1\times d^2)} & \delta_n^c \times \Delta[2] \ar@{.>}[rrrrru]^G\ar[rrrrr]_{\ct_{p(z)} = \ct_{p(w)} }  &&& &&   B
}
$$ and so, by Corollary~\ref{L:Lifting-Property} and the fact that $p:X\rightarrow B$ is a fibration, there is a lift $G$ and the composition $G\circ (1\times d^2)$ is a $p$-homotopy from $\beta(w)$ to
$\beta(z)$.

\noindent\textbf{Step 3. } The map $\alpha$ is injective in dimension $n$.

Let $x$ and $y$ be $n$-simplices of $Z$ such that $\alpha(x)=\alpha(y)$. Since $\beta\colon Z\Right0{} X$ is an isomorphism, $Z$ is also a free $\calc$-diagram with
base set $\Sigma(Z)= \beta^{-1}(\Sigma(X))$. Then, there are  $x_0, y_0\in \Sigma(Z)$ and $f,g$ in $\calc$ such that
$$f(x_0) =x\quad \text{and} \quad g(y_0)=y.$$

Notice that $f(\alpha(x_0)) = \alpha(x)=\alpha(y)= g(\alpha(y_0))$, so there exists $w\in \Sigma(X)$ and morphisms
$h, l$ in $\calc$ such that
$$h(w) = \alpha(x_0),\textrm{ } l(w) = \alpha(y_0)   \quad \text{and} \quad  f\circ h =  g\circ l\,.$$

By Step~1, $\alpha$ is surjective in dimension $n$, so there exists $v\in Z$ such that $\alpha(v)=w$, and thus we have $\alpha(h(v)) = h(\alpha(v))=h(w) = \alpha(x_0)$ and
$\alpha(l(v)) = l(\alpha(v))= l(w) = \alpha(y_0)$. Therefore, we can apply Step~2 in order to conclude that
$$
 \beta(h(v) )  \simeq_p   \beta(x_0)  \quad \text{and} \quad  \beta(l(v))  \simeq_p  \beta(y_0)\,.
$$

Now, there is also a simplex $u\in \Sigma(X)$ and a morphism $r$ in $\calc$ such that $r(u) = \beta(v) $. Then $h(r(u)) = h(\beta(v)) = \beta(h(v)) \simeq_p g(x_0)$, and
also $l(r(u)) = l(\beta(v)) = \beta(l(v)) \simeq_p \beta(y_0)$, so we have shown that
$$
u\preceq \beta(x_0)  \quad \text{and} \quad   u\preceq \beta(y_0)\,.
$$

However, $x_0, y_0\in \Sigma(Z)$, and hence the three simplices $u, \beta(x_0), \beta(y_0)$ are in $\Sigma(X)$. Since $p\colon X\Right0{}B$ is minimal,
$\beta(x_0)= u =\beta(y_0)$ and $h\circ r = l\circ r$ must be an endomorphism, and then an automorphism of $\calc$. Furthermore, $x_0=y_0$ since $\beta$ is
isomorphism. Finally,  we have $f\circ(h\circ r) =  (f\circ h)\circ r = (g\circ l)\circ r = g\circ(l \circ r) $, thus $f = g$. So we obtain
$x= f(x_0) = g(y_0)= y$.
\end{proof}

\begin{Nota}
 It is not true in general that a minimal fibration in $\mathcal{S}^{\mathcal{C}}$ is minimal
objectwise, that is, if $p\colon X\Right0{} B$ is a minimal fibration of $\calc$-diagrams,
then it is not always true that
for every object $a$ of $\calc$, $p_a\colon X_a\Right0{} B_a$
is a minimal fibration in $\mathbf{S}$. The following is an illustrative example.
Consider the category $\mathcal{C}=\{a\Right0{f}  b\Left0{g}c\}$ and
a minimal fibration $p\colon X\Right0{} B$, where $X$ is a free diagram. In this case
$p_a\colon X_a\Right0{} B_a$ and $p_c\colon X_c\Right0{} B_c$ are minimal, but there might exist generators $z\in X_a$ and
$w\in X_c$ such that $f(z)\neq g(w)$ and $f(z)\simeq_{p_b} g(w)$, so $p_a\colon X_a\Right0{} B_a$ would not be a minimal fibration of
simplicial sets.
\end{Nota}

\begin{Cor}\label{C:Isomorphism between minimal fibrations} Any two minimal fibrations over a diagram $B$ that are fibrewise homotopy equivalent, are isomorphic.
\end{Cor}

\begin{proof}
Let $p:X\rightarrow B$ and $q:Y\rightarrow B$ be minimal fibrations over $B$ and let
$\alpha\colon X\rightarrow Y$ and $\beta\colon Y\rightarrow X$ be fibrewise homotopy inverses.
The composition $\beta\circ\alpha$ is then fibrewise homotopic to the identity map of $X$,
and so Proposition \ref{P:strong-isom-min.f} implies that $\beta\circ\alpha$ is an isomorphism.
A similar argument shows that $\alpha\circ\beta$ is also an isomorphism, hence both $\alpha$ and
$\beta$ are isomorphisms over $B$.
\end{proof}

We end this section with an example showing how Theorem~\ref{T:const-min-fibrat(A)}
fails on non-artinian index categories.

\begin{Ej}\label{ExampleNop}
 Write $\caln$ for the poset of non-positive integers
 $\con{\dots < n< n+1 < \dots <-1<0}$. A $\caln$-diagram $X$ is a sequence
 of spaces and maps
$$\cdots \Right2{f_n} X_{n} \Right2{f_{n+1}} X_{n+1} \Right2{f_{n+2}} \cdots \Right2{f_{-1}} X_{-1} \Right2{f_0}  X_0 \,.$$
Assume that each $X_i$ is non-empty, fibrant and connected, and that each $f_n$ is injective. Then
\begin{enumerate}
\item $X$ is a fibrant and it is a free diagram if and only if each simplex has only a finite number of preimages.
\item $X$ is never minimal and it does not contain a minimal retract.
\end{enumerate}

In any case $X$
is fibrant since each $X_i$ is a fibrant simplicial set.

Define $\Sigma$ as the union of the sets $\Sigma_n$, $n\leq 0$, where $\Sigma_n$ is the set of simplices of $X_n$
that do not belong to the image of $f_n\colon X_{n-1} \Right0{} X_n$. If every simplex has a finite number of preimages
in the diagram, then every simplex is the image of an element of $\Sigma$ by a morphism of the index category; as since
each $f_i$ is injective, $\Sigma$ is a base  and $X$ is a free diagram.
If for some $n\leq 0$ there is a simplex
$x\in X_n$ which is in the image of $f_n\circ \dots \circ f_m$ for all $m< n$, then $X$ is not free.
If it was free with base $\Sigma'$, then there would be  a generator  $y\in \Sigma'$ with
$f_n\circ \dots \circ f_{r+1}(y)=x$, for $y\in X_r$. In this situation we could find a preimage $x'$ of $x$ in $X_{r-1}$, and hence
there would be another generator
$z\in \Sigma'$, $z\in X_s$ for some $s>r$ such that $f_{r-1}\circ \dots \circ f_{s+1}(z)=x'$. But, since each $f_i$ is injective,
we would also have $f_n\circ \dots \circ f_{r}\circ\dots \circ f_{s+1}(z)= x$
and this would contradict uniqueness of generators mapping to a fixed simplex.

 Assuming that $X$ is free, we show that it cannot be minimal.
Suppose on the contrary that it was minimal.
A base $\Sigma$ must contain 0-simplices, so for a certain $n$ there exists a
0-simplex $x$ in $\Sigma_n$. But since the precedent $X_i$, $i<n$,
are non-empty, there must be another $m<n$ and a 0-simplex $y$ in $\Sigma_m$. Since
$f_n\circ \dots \circ f_{m+1}(y)\in X_n$, and $X_n$ is connected, we have $y\preceq x$
and this would imply $x=y$, that contradicts $m<n$.
\end{Ej}

\section{$\mathcal{C}$-twisted cartesian products and $\mathcal{C}$-fibre bundles\label{S:C-fibre bundles}}
This section contains the proof of Theorem~\ref{T:Classif-C-fibre(B)}.
After generalizing the concepts of twisted cartesian product (TCP) and fibre bundle to the category of $\calc$-diagrams,
we show that the classification of $\calc$-fibre bundles reduces to the classification of principal $G$-fibre bundles,
always with constant base $\calc$-diagram.
The theory of minimal fibrations developed in Section~\ref{MinFib}  applies then to show that the classification of fibrations of
$\calc$-diagrams over constant base, where $\calc$ is an artinian EI-category, reduces to that of $\calc$-fibre bundles
over $B$.

Recall that a  simplicial group $G$ is a functor $G\colon \Delta\op\Right0{}\gr$. It can be seen as a sequence
of groups $\{G_n\}_{n\geq0}$ together with face and degeneracy operators subject to the same
relations~\eqref{simprel} as simplicial sets.
The composition
with the underlying functor $\calu\colon\gr\Right0{}\Sets$ gives the underlying simplicial set of a
simplicial group.

\begin{Def}
\label{leftaction} Let $\calc$ be a small category.
A \emph{left action} of a simplicial group $G$ on a $\calc$-diagram $F$ is a map of $\calc$-diagrams
\begin{equation*}
 t \colon G\times F \Right2{} F
\end{equation*}
denoted $t(g,x)=g\cdot x$, that satisfies the usual axioms for a group action, namely,
for any $n\geq0$, $c\in \Ob(\calc)$, $g,h\in G_n$ and
$x\in F_{c,n}$, $g\cdot(h\cdot x)= (gh)\cdot x$ and $e_n\cdot x=x$ if $e_n$
denotes the unit of $G_n$.
\end{Def}

Given a $\calc$-diagram $F$, the simplicial group of automorphisms of $F$ is the subsimplicial set $\aut_\calc(F)$ of $\map_\calc(F,F)$ which consists of the commutative diagrams
$$\xymatrix{ F\times \Delta[n]\ar[rr]^{\5\varepsilon} \ar[dr]_{pr} &&
                 F\times \Delta[n]\ar[dl]^{pr}  \\      & \Delta[n]   }
$$
in which the horizontal map is invertible. It clearly becomes a simplicial group that acts on $F$ by evaluation:
\begin{equation*}
   \ev\colon \aut_\calc(F) \times F \Right2{} F\,.
\end{equation*}
That is, given $\5\varepsilon$ and a simplex $x\colon \delta_n^c\Right0{}F$, $\ev(\5\varepsilon,x) =
\varepsilon (x,\iota_n) $, where $\varepsilon$ is the composite of $\5\varepsilon$ and the projection
$F\times\Delta[n]\Right0{}F$.

Sometimes we express the action in its adjoint form $\rho\colon G\Right0{}\autc(F)$.
 Notice that in that way, an element $g\in G_n$ determines an isomorphism of $\calc$-diagrams
 $\rho(g)\colon F\times\Delta[n]\Right0{}F\times\Delta[n]$ over $\Delta[n]$. Now, given $g\in G_n$ and $
 x\in F_n$ with characteristic maps $g\colon \Delta[n]\Right0{}G$ and $x\colon \delta_n^c\Right0{}F$, the action
 $g\cdot x$ is characterized by  the composition
 $$
    \delta_n^c \Right2{d}  \Delta[n] \times  \delta_n^c \Right2{g\times x}  G\times F  \Right2{\cdot} F\,.
 $$

The classical notion of twisting function will be crucial in the sequel.
\begin{Def}[{\cite[\S18]{May}}]
\label{twisting}
Let $G$ be a simplicial group and  $B$ a simplicial set. Then
a \emph{twisting function} $t\colon B\Right0{} G$ is a collection of functions $\{t_{n}\colon B_{n}\Right0{} G_{n-1}\}_{n\geq1}$ that satisfy:
\begin{align*}
 d_it_{n+1}(v) & =t_{n}(d_{i+1}v), & & i\geq 1, \\
 s_it_{n}(v)      &  =t_{n+1}(s_{i+1}v), & &  i\geq 0,\\
 d_0t_{n+1}(v)   & =   [t_{n}(d_0v)] ^{-1}\cdot t_{n}(d_1v) ,\\
 t_{n+1}(s_0v)& =e_{n}\,,
\end{align*}
for all $n\geq1$.
\end{Def}

Now we have the ingredients we needed in order to define $\mathcal{C}$-twisted
 cartesian products:

 \begin{Def}\label{D:C-TCP} Let  $G$ be a simplicial group and $F$  a $\mathcal{C}$-diagram  with a left action of $G$.
 Let $B$ be a simplicial set and $t\colon B\Right0{} G$  a twisting function, then the \emph{$\mathcal{C}$-twisted
 cartesian product} $B\times_t F$ is the
$\mathcal{C}$-diagram with simplices
$$
      (B\times_t F)_{c,n} = B_n \times F_{c,n}
$$
and structural operators
\begin{align*}
 d_0(b,x)  & = (d_0(b) , t(b)\cdot d_0(x)) \\
 d_i(b,x)     &  =   (d_ib , d_ix )   , & &  i\geq 1,\\
s_i (b,x) & = (s_ib, s_ix)    , & &  i\geq 0 \,,\\
f(b,x) & = (b,f(x))\,.
\end{align*}

\end{Def}

For short, we will frequently refer to a $\mathcal{C}$-twisted
 cartesian product simply as a $\mathcal{C}$-TCP.

Once we have extended the concept of twisted cartesian product, we have to deal with fibre bundles. To develop this notion in the context of
$\mathcal{C}$-diagrams of spaces we will focus on maps $p\colon X\Right0{} B$ of $\mathbf{S}^{\mathcal{C}}$, where $B$ is a constant diagram to the
simplicial set $B$. Under this assumption, the classical concepts about local triviality, atlases and structural group (see \cite[11.8]{May}) are generalized, in such a
way that we can recover again the theory of $\calc$-twisted cartesian products over a constant base space.

\begin{Def}\label{D:C-fibre-b} Let $F$ be a $\mathcal{C}$-diagram and $B$ a simplicial set. A map $p\colon X\Right0{} B$ in $\mathbf{S}^{\mathcal{C}}$, where
$B$ is regarded as a constant diagram, will be called a \emph{$\mathcal{C}$-fibre bundle} with fibre $F$ if $p$ is an epimorphism and for every n-simplex $v\in B$
there exists an isomorphism of $\calc$-diagrams, $\alpha_p(v)\colon \Delta [n]\times F\Right0{} \Delta [n]\times_{B} X$, such that the following diagram commutes:
\index{$\mathcal{C}$-fibre bundle}
\begin{equation}\label{alpha-C-atlas}
 \xymatrix{  \Delta [n]\times F\ar[r]^(0.5){\alpha_p(v)}_{\cong}\ar[dr]_{\proj}  &  \Delta [n]\times_{B} X\ar[r]^(0.65){\5{v}}\ar[d]^{\5p}   &    X\ar[d]^p      \\
                                   &  \Delta [n]\ar[r]_{v}                   &    B\ar@{}[ul]|(0.2){\lrcorner}\rlap{\,.}  }
\end{equation}

A set of isomorphisms $\conj{\alpha_p(v)}{v\in B_n\,,\ n\geq0}$ making commutative
the diagram \eqref{alpha-C-atlas}
for each simplex $v$ of $B$,  is called a \emph{$\mathcal{C}$-atlas} for $p$.
\end{Def}

Set also $\beta_p(v) = \5v\circ \alpha_p(v)$.
Notice that the elements $\alpha_p(v)$ are isomorphisms over $\Delta[n]$,
while $\beta_p(v)\in \mapc(F, X)_n$ are injections such that the diagram
\begin{equation}\label{beta-C-atlas}
\xymatrix{  \Delta [n]\times F\ar[r]^-{\beta_p(v)}\ar[d]_{\proj}  &    X\ar[d]^p      \\
                                  \Delta [n]\ar[r]_{v}                   &    B  }
\end{equation}
is a pullback diagram of $\calc$-diagrams.

Conversely, given a map $\beta_p(v)$ making
diagram~\eqref{beta-C-atlas} commutative and a pullback diagram, there is a unique
factorization  $\beta_p(v) =  \5v\circ \alpha_p(v)$, where
$\alpha_p(v) \colon  \Delta [n]\times F\Right0{}\Delta [n]\times_{B} X$
is an isomorphism making \eqref{alpha-C-atlas} commutative, since the right
hand square of diagram \eqref{alpha-C-atlas} is a pullback diagram.

It follows that sets of maps  $\{\alpha_p(v)\}$ and
$\{\beta_p(v)\}$ determine each other.  When talking about atlases we will choose
$\{\alpha_p(v)\}$ or $\{\beta_p(v)\}$ depending on the context, without further
explanation when there is no possible confusion.

Given two atlases $\{\alpha_p(v)\}$ and
$\{\widetilde{\alpha}_p(v)\}$ of $p$, $\alpha_p(v)^{-1}\widetilde{\alpha}_p(v)\in\autc(F)_n$,
and conversely, if for every $v\in B_n$ we choose
$(\gamma(v))\in \autc (F)_n$, then $\{\alpha_p(v)\gamma(v)\}$ is another atlas.

In the sequel we translate to the context of diagrams some classical notions of Bundle Theory.

\begin{Def}\label{normalatlas}
In the previous notation, the atlas is \emph{normal} provided
$\beta_p(s_iv)=s_i\beta_p(v)\in \Hom_\calc(F,X)_{n+1}$,
for each simplex $v$ of $B_n$,  $0\leq i\leq n$.
\end{Def}

Notice that this equality holds exactly when
$\alpha_p(v)\circ (s^i\times 1) = (\5s^i\times1)\circ \alpha_p(s_iv)$.
In other words, the top left square of the following diagram commutes:
$$
    \xymatrix{    &  \Delta [n]\times F\ar[r]^(0.55){\alpha_p(v)}_(0.5){\cong}
                                     & \Delta [n]\times_B X\ar[r]\ar@{.>}[d]  &  X\ar[d]^p  \\
              \Delta [n+1]\times F\ar[ur]^{s^{i}\times 1}\ar[r]^{\alpha_p(s_iv)}_(0.5){\cong}  &  \Delta [n+1]\times_B X\ar[ur]^(0.4){\widetilde{s}^i\times 1}\ar[urr]\ar[d]   & \Delta [n]\ar[r]^{v} &  B.  \\
                  &  \Delta [n+1]\ar[ur]^{s^i}\ar[urr]_{s_iv}     &     &     }
$$
If  $\{\beta_p(v)\}$ is an atlas, we can always redefine it
over the degenerate simplices using the following definition dimensionwise
$$
\begin{cases}
  \beta'_p(s_iv):=s_i\beta_p(v) \,,   & v\in B_n\,,\ 0\leq i\leq n\,, \\
  \beta'_p(v):=\beta_p(v) \,,    & \text{if $v$ is non-degenerate,}
\end{cases}
$$
so it is shown by induction on $n$ that the new atlas
$\{\beta'_p(v)\}$ is well-defined and normal. Indeed,
it is clearly well-defined on $0$-simplices. Assume that it is well-defined on $n$-simplices and that
we have two expressions $s_iv=s_jv'$ for a degenerate simplex of dimension $n+1$. It follows
that $v'= s_id_j(v)$ and then $s_j\beta_p(v') = s_i\beta_p(v)$.
Thus, upon replacing $\{\beta_p(v)\}$ by  $\{\beta'_p(v)\}$ we can
assume that a given atlas is normal.
(cf.~\cite[pg.~648]{Barrat}, \cite[6.5]{Curtis}, \cite[\S19]{May}).

We turn now our attention to face operators. In general
$\beta_p(d_iv)$ and $d_i\beta_p(v)$ will differ by an automorphism
$\xi_p^i(v)\in\autc(F)_{n-1}$,  $d_i\beta_p(v)  =  \beta_p(d_iv)\circ \xi_p^i(v)$,  as it is shown
in the following commutative diagram:
\begin{equation}\label{G-atlas}
 \xymatrix{     &   \Delta [n-1]\times F\ar[r]^{d^{i}\times 1}    \ar@{.>}[ld]_{\xi^i_p(v) }
                      &  \Delta [n]\times F\ar[r]^(0.55){\alpha_p(v)}_(0.5){\cong}
                      &  \Delta [n]\times_B X\ar[r]\ar[d] |\hole&  X\ar[d]^p  \\
              \Delta [n-1]\times F\ar[rr]^{\alpha_p(d_iv)}_(0.5){\cong} &
              &  \Delta [n-1]\times_B X\ar[ur]^(0.4){\widetilde{d}^i\times 1}\ar@/_1.5mm/[urr]
                     \ar[d]   & \Delta [n]\ar[r]^{v} &  B.  \\
                    &            &  \Delta [n-1]\ar[ur]^{d^i}\ar@/_1.5mm/[urr]_{d_iv}  & & }
\end{equation}

We will refer to $\conj{\xi_p^i(v)}{v\in B_n\,,\ 0\leq i\leq n}$ as the set of \emph{transformation elements} associated to the atlas $\{\alpha_p(v)\}$.
The following notions generalize classical properties of the atlases:

\begin{Def}[{\cite[IV.2.4]{Barrat},\cite[6.5]{Curtis},\cite[19.1]{May}}]
\begin{enumerate}[\rm (1)]
\item An atlas $\{\alpha_p(v)\}$ is \emph{regular} if for every $v\in B_n$ and $1\leq i\leq n$,
$\xi_p^i(v)=e_{n-1}$, where $e_{n-1}$ is the identity element of $\autc(F)_{n-1}$.
\item Assume that $G$ is a subgroup complex of $\autc(F)$.
We will say that $\{\alpha_p(v)\}$ is a \emph{$G$-atlas} if for each simplex $v$ of
$B$, $\xi^i_p(v)\in G$.
\item Two $G$-atlases $\{\alpha_p(v)\}$ and  $\{\widetilde{\alpha}_p(v)\}$ are
\emph{$G$-equivalent} if for each $v\in  B$,
there  exists $\gamma(v)\in G$ such that $\widetilde{\alpha}_p(v)=\alpha_p(v)\gamma(v)$.
\end{enumerate}
\end{Def}

Now we can finally define $G$-$\calc$-fibre bundles:

\begin{Def}
A \emph{$G$-$\calc$-fibre bundle} is a
$\calc$-fibre bundle together with a $G$-equivalence class of $G$-atlases.
We call $G$ the \emph{structural group} of the $\calc$-fibre bundle.
\end{Def}

Notice that the regularity condition $\xi_p^i(v)=e_{n-1}$ is equivalent to
$\beta_p(d_iv)=d_i\beta_p(v)$, for $i>0$.

In order to work with nice atlases in our $G$-$\calc$-fibre bundles,
we need the following statements, that in turn generalize classical results with essentially the same proof.

\begin{Pro}[{\cite[2.5]{Barrat},\cite[6.6]{Curtis},\cite[19.2]{May}}]\label{reg. atlases}
Every equivalence class of $G$-atlases contains a regular $G$-atlas.
\end{Pro}

\begin{Pro}[{\cite[19.4]{May}}]\label{trans-elemts}
The transformation elements $\{\xi_p^0(v)\}$ of a regular $G$-atlas define a twisting function
$ \xi^0_p\colon B\Right0{}G$.
\end{Pro}

From now on, given any $G$-$\mathcal{C}$-bundle we may suppose that we are dealing
with regular and normalized atlases. The following definition is a
generalization of the concept of $G$-map and $G$-equivalence (cf.~\cite[19.1]{May})
to the context of $G$-$\mathcal{C}$-fibre bundles.

\begin{Def}\label{defGequiv}
If $p\colon X\Right0{} B$ and $p'\colon X'\Right0{} B$
are $G$-$\mathcal{C}$-fibre bundles with fibre $F$, a $G$-\emph{map} from $p$ to $p'$ is a map of
$\mathcal{C}$-diagrams $h\colon X\Right0{} X'$ such that
given $G$-atlases $\{\alpha_p(v)\}$, $\{\alpha_{p'}(v)\}$ of $p$ and $p'$, respectively,
there exists $\gamma(v)\in G_n$ such that
$h\beta_p(v)=\beta_{p'}(v)\gamma(v)$; that is, the  following diagram commutes:
$$
\xymatrix{   &  \Delta [n]\times F\ar[r]^{\alpha_{p'}(v)}\ar@{..>}[ddl]
                  &  \Delta [n]\times_B X'\ar@{..>}[ddl]\ar[r]^(0.65){\widehat{v}'}
                  &  X'\ar[ddl]^{p'}  \\
      \Delta [n]\times F\ar[ur]^{\gamma(v)}\ar[d]\ar[r]^{\alpha_p(v)}
                      &  \Delta [n]\times_B X\ar[r]^(0.55){\widehat{v}}\ar[d]
                      &  X\ar[d]^p\ar[ur]^h     &      \\
       \Delta [n]\ar[r]_{-}     &  \Delta [n]\ar[r]_{v}      &  B     &   }
$$
Here we write $\beta_p(v) = \5v\circ \alpha_p(v)$ and $\beta'_p(v) = \5v'\circ \alpha'_p(v)$, as usual.

If $h$ is a natural isomorphism we will say that $p$ and $p'$ are \emph{$G$-equivalent.}
When $G= \autc(F)$, an isomorphism $h\colon  X\Right0{}X'$ over $B$ is automatically
an $\autc(f)$-equivalence.  In this case we say that the two $\calc$- fibre
bundles are equivalent, with no further mention of the structural group.
\end{Def}

Recall (cf.\ \cite[\S20]{May}) that two twisting functions
$t,t'\colon B\Right0{}G$ are $G$-equivalent if there is a degree-preserving function
$\gamma\colon B\Right0{}G$ such that
\begin{align}\label{eq-twisting-functions}
	t'(v)\cdot d_0\gamma(v) & =  \gamma(d_0v)\cdot  t(v) \,, \\
\label{eq-twisting-functions2}    d_i\gamma(v) & =  \gamma(d_iv)\,, & & i>0\,,\\
	\label{eq-twisting-functions3}	    s_i\gamma(v) & =  \gamma(s_iv)\,, & & i\geq0\,,
\end{align}
and this is an equivalence relation.

\begin{Pro}\label{fibrequivalent} Two $G$-$\calc$-fibre bundles are $G$-equivalent
if and only if the corresponding twisting functions are $G$-equivalent.
 \end{Pro}
\begin{proof}Let $p_1\colon X_1\Right0{} B$ and $p_2\colon X_2\Right0{} B$
be two $G$-$\calc$-fibre bundles. Choose $G$-atlases $\{\alpha_1(v)\}$ and
($\{\beta_1(v)\}$), and $\{\alpha_2(v)\}$ and ($\{\beta_2(v)\}$)  for
$p_1$ and $p_2$ respectively.

 If $p_1$ and $p_2$ are  $G$-equivalent, then by definition
 there is a function $\gamma\colon B\Right0{} G$ with
 $h\circ \beta_2(v) = \beta_1(v)\circ \gamma(v)$,
 for each simplex $v$ of $B$.
 Then, if $t_1$, $t_2$ denote the respective twisting functions, we have $\beta_i(d_0v)\circ t_i(v) =d_0(\beta_i(v))$,
 $i=1,2$. Hence,
\begin{align*}
 \beta_1(d_0v)\circ \gamma(d_0v)\circ t_2(v) & = h\circ \beta_2(d_0v)\circ t_2(v) \\
                           & = h\circ d_0(\beta_2(v)) \\
                            & = d_0(h\circ \beta_2(v)) \\
                           & = d_0\bigl( \beta_1(v)\circ \gamma(v)\bigr) \\
                           & = d_0( \beta_1(v)) \circ d_0(\gamma(v)) = \beta_1(d_0v)\circ t_1(v)\circ d_0(\gamma(v))\,,
\end{align*}
 and since $ \beta_1(d_0v)$ is injective $\gamma(d_0v)\circ t_2(v) = t_1(v)\circ d_0(\gamma(v))$.
 This is the first condition above \eqref{eq-twisting-functions} and the other two conditions follow similarly.
 Therefore, $t_1$ and $t_2$ are $G$-equivalent.

 Reading in reverse, the same argument proves that if $t_1$ and $t_2$ are $G$-equivalent, then $p_1$ and $p_2$ are
 $G$-equivalent $G$-$\calc$-fibre bundles.
 \end{proof}

The above notion of $G$-equivalence is translated to the case of $\calc$-TCP's as follows.

\begin{Def}
 Two $\calc$-TCP's, $B\times_tF$ and
	$B\times_{t'}F$, with
	structural group $G$ are
	$G$-\emph{isomorphic} if there is a function $\gamma\colon B\Right0{} G$ such that
	$h\colon B\times_tF \Right0{}B\times_{t'}F$ defined by $h(v,x) = (b, \gamma(v)x)$ is an isomorphism
	of simplicial sets.
\end{Def}

The twisting functions codify, in fact, the notion of $G$-isomorphism between $\calc$-TCP's:

\begin{Pro}\label{D:Map of C-TCP's}
	Two $\calc$-TCP's $B\times_tF$ and
	$B\times_{t'}F$ with
	structural group $G$ are
	isomorphic over $B$ if and only if the twisting functions
	$t$ and $t'$ are equivalent.
\end{Pro}
 \begin{proof} It is easily checked that a function $\gamma\colon B\Right0{} G$ satisfies
 conditions \eqref{eq-twisting-functions},  \eqref{eq-twisting-functions2}, and
 \eqref{eq-twisting-functions3} if and only if $h\colon B\times_tF\Right0{}B\times_{t'}F$
 defined as $h(v,x) =(v,\gamma(v)x)$ is simplicial.
\end{proof}

\begin{Pro}\label{psi-surj}
Let $p\colon X\Right0{} B$ be a $\mathcal{C}$-fibre bundle with fibre $F$ and regular $G$-atlas $\{\alpha_p(v)\}$. Then the transformation
elements $\{\xi^0_p(v)\}$ define a $\mathcal{C}$-twisting map $\xi^0_p\colon B_n\Right0{} G_{n-1}$ and thereby $B\times_{\xi^0_p} F$ becomes a
$\mathcal{C}$-$TCP$ with fibre $F$ and group $G$. Furthermore there is an isomorphism $h\colon B\times_{\xi^0_p} F\Right0{} X$ of $\mathcal{C}$-fibre bundles
with group $G$.

\end{Pro}

\begin{proof} By Proposition~\ref{trans-elemts}, $\xi^0_p$ is a twisting function.
We can now form the $\mathcal{C}$-TCP $B\times_{\xi^0_p} G$ with a projection $B\times_{\xi^0_p} G\Right0{} B$ which
is indeed a principal $G$-fibre bundle. Then, we obtain a new $\calc$-fibre bundle as $B\times_{\xi^0_p} F \cong
(B\times_{\xi^0_p} G)\times_G F \Right0{} B$. We will now construct a $G$-equivalence
$h \colon B\times_{\xi^0_p} F \Right0{} X$ over $B$.

For any $n\geq0$, $c \in \Ob(\calc)$, $v\in B_n$, and $z\in F_{c,n}$, define
$h_c(v,z)=\beta_{p_c}(v)(\imath_n,z)\in X_{c,n}$.
By Proposition \ref{trans-elemts},  $h_c$ is an isomorphism for every $c\in\mathcal{C}$.
It
remains to show that for every $f\colon c\Right0{} d\in\mathcal{C}$ the following square commutes:
$$
\xymatrix{B\times_{\xi^0_{p_c}} F_c\ar[r]^(0.6){h_c}\ar[d]_{1\times F_f}    &   X_c\ar[d]^{X_f}   \\
            B\times_{\xi^0_{p_d}} F_d\ar[r]_(0.6){h_d}                        &   X_d.   }
$$
Take an $n$-simplex $(v,z)$ in $B\times_{\xi^0_{p_c}} F_c$ and evaluate it in
the above square, that is, $X_fh_c(v,z)=X_f\beta_{p_c}(v)(\imath_n,z)$
and $h_d(1\times F_f)(v,z)=h_d(v,F_f(z))=\beta_{p_d}(v)(\imath_n,F_f(z))$.
Since $p$ is a $\mathcal{C}$-fibre bundle it looks locally as the diagram $(7)$,
which is commutative, and therefore
$X_f\beta_{p_c}(v)(\imath_n,z)=\beta_{p_d}(v)(1\times F_f)(\imath_n,z)$, where
$\beta_{p_d}(v)(1\times F_f)(\imath_n,z)=\beta_{p_d}(v)(\imath_n,F_f(z))$.
\end{proof}

Now we can prove the key correspondence between principal $G$-fibre bundles
and $G$-$\calc$-fibre bundles:

\begin{Teo}\label{Tveo} Let $\calc$ be a small category,  $F$ a $\calc$-diagram, and $G$ a simplicial group
with a left action on $F$.  Given a simplicial set $B$, there is a bijection of sets
\begin{equation*}
 \left\{  \parbox{4cm}{\flushleft Isomorphism classes of principal $G$-fibre bundles over $B$}      \right\}
    \RIGHT6{\Psi}{\cong}
 \left\{  \parbox{4cm}{\flushleft $G$-equivalence classes
of $G$-$\calc$-fibre bundles over $B$ with fibre $F$}      \right\}
\end{equation*}
which assigns the $G$-class of $G$-$\calc$-fibre bundles
represented by $p=\xi\times_G1\colon E\times_G F\Right0{} B$ to the class of a principal $G$-fibre bundle
 $\xi\colon E\Right0{}B$.
\end{Teo}
\begin{proof}
We first check that $\Psi$ is well-defined. A principal $G$-fibre bundle $\xi\colon E\Right0{}B$
is locally trivial, hence $p\colon E\times_G F\Right0{} B$ is also locally trivial, and then
a $\calc$-fibre bundle. Furthermore, the transformation elements of $\xi$ belong to $G$ and
provide a $G$-atlas for $p$, hence $p$ is really a $G$-$\calc$-fibre bundle. Finally, if $\xi$ and $\xi'$
are isomorphic principal $G$-fibre bundles, they have the same twisting functions by
Proposition~\ref{fibrequivalent}, and then Proposition \ref{D:Map of C-TCP's}
implies that $p$ and $p'$ are $G$-equivalent.

 We state now the surjectivity of $\Psi$. Given a $G$-$\calc$-fibre bundle $p\colon X\Right0{}B$, pick a regular $G$-atlas,
 so the transformation elements form a $G$-twisting function. This function
 determines a $\mathcal{C}$-TCP $\xi\colon B\times_t G\Right0{}B$, which is a principal $G$-fibre bundle. Then we should check
 that $\Psi([\xi]) = [p]$, that is, that there is a $G$-equivalence over $B$:
  $(B\times_t G)\times_G F \Right0{\cong} X$. But this is Proposition \ref{psi-surj}.

Finally, let us check that $\Psi$ is injective. Let $[\xi_1]$ and $[\xi_2]$ be two isomorphism classes of
 principal $G$-fibre bundles such that $\Psi([\xi_1])=\Psi([\xi_2])$. Fix representatives of these isomorphism classes
 that are $\mathcal{C}$-TCP's,  $\xi_i\colon B\times_{t_i} G\Right0{} B$, $i=1,2$, with twisting
 functions $t_i\colon B\Right0{}G$, $i=1,2$. Then, for $i=1,2$,
$\Psi([\xi_i])$ are represented by the $\calc$-TCP's $B\times_{t_i}F\cong (B\times_{t_i} G)\times_G F$.
So, $\Psi([\xi_1])=\Psi([\xi_2])$ if and only if $B\times_{t_1}F$ and $B\times_{t_2}F$ are $G$-equivalent.
By Proposition \ref{D:Map of C-TCP's}  this
happens if and only if the twisting functions $t_1$ and $t_2$ are $G$-equivalent. But Proposition \ref{fibrequivalent} implies that
this is equivalent to
 $B\times_{t_1}G$
and $B\times_{t_2}G$ being  isomorphic principal $G$-fibre bundles. In particular, $[\xi_1] = [\xi_2]$.
 \end{proof}

\begin{proof}[Proof of Theorem~\ref{T:Classif-C-fibre(B)}]
This is now an immediate consequence of Theorem~\ref{Tveo} and the classical classification of isomorphism classes of principal $G$-bundles.
\end{proof}

\section{Classification of fibrations}
\label{Classifibrations}
We proceed now to the proof of Theorem~\ref{T:Clas-C-Fib-1(C)}. The index category $\calc$ will be assumed to be
a small artinian $EI$-category.

\begin{Lem}\label{L:pullback of min.f} Let $\mathcal{C}$ be a small category,
$\alpha\colon A\Right0{} B$ a map in $\mathbf{S}$ and $p\colon X\Right0{} B$ a minimal
fibration in $\mathbf{S}^{\mathcal{C}}$ over the constant $\mathcal{C}$-diagram to $B$.
Then, the pullback fibration  $\5p\colon A\times_B X\Right0{} A$ is also minimal.
\end{Lem}

\begin{proof} Let  $\Sigma$ be a basis for $X$ in the sense of Definition~\ref{P:Basis}.
We will show that $\Sigma^{^\sqcap}=\conj{(u,x)\in A\times_B X}{x\in\Sigma}$ is
a basis for $A\times_B X$. Pick an arbitrary element $(u,x)\in \Sigma^{^\sqcap}$,
so $u\in A$ and $x\in \Sigma$. Since $\Sigma$ is closed under degeneracy operators,
for each $k$, $s_kx\in \Sigma$ and then also
$s_k(u,x)=(s_ku,s_kx)\in\Sigma^{^\sqcap}$, thus $\Sigma^{^\sqcap}$ is also closed
under degeneracy operators.

Pick now an arbitrary element $(v,z)\in A\times_B X$; now $z\in X$ and there is a
 unique morphism $f$ of $\calc$ and a unique
$x\in \Sigma$ such that $f(x)=z$.
Then, $(v,x)\in  A\times_B X$ for $p(x)=(p\circ f)(x)=p(z)=\alpha(v)$,
so in fact  $(v,x)\in \Sigma^{^\sqcap}$, and $f(v,x)=(v,z)$.
If there were another $g\in \Mor(\calc)$
and another $(w,y)\in\Sigma^{^\sqcap}$ such that
$g(w,y)=(v,z)$, we would have  that $(w,g(y))=(v,z)=(v,f(x))$, so $w=v$, and
since $\Sigma$ is a basis for $X$,  $g=f$ and $y=x$. We have proved that $A\times_B X$
is a free $\calc$-diagram with base $\Sigma^{^\sqcap}$.

In order to prove that the fibration $\5p\colon A\times_B X\Right0{} A$ is minimal it remain to show that
given two elements of a given base for $A\times_B X$, if one preceds the other then they coincide.
Assume that  $(v,z)$ and $(w,x)$ are elements of $\Sigma^{^\sqcap}$ and
suppose that $(w,x)\preceq(v,z)$. Then there exists a map $f\in\mathcal{C}$ such that
$f(w,x)\simeq_{p} (v,z)$, so $w=v$ since $p$-homotopic simplices are in the same fibre.
We also have that $x\simeq_p z$, but from minimality of $p$ it must hold
that $x=z$, that is, $(v,z)=(w,x)$.
\end{proof}

The proof of the following lemma is analogous to that of Lemma 10.6 in Chapter I of \cite{Goerss-Jardine}.
In the case of $\mathbf{S}^{\mathcal{C}}$ the lifting properties used in this
proof can be obtained using Proposition~\ref{cof}
and Axiom M4 (see Definition~\ref{functcomp}).
This result will help to establish a connection
between the theory of minimal fibrations and the theory of $\mathcal{C}$-fibre bundles.

\begin{Lem}\label{L:f.h.e} Let $\mathcal{C}$ be a small category, $\alpha_0,\alpha_1\colon A\Right0{} B$ homotopic maps in $\mathbf{S}^{\mathcal{C}}$ and let
$p\colon X\Right0{} B$ be a fibration in $\mathbf{S}^{\mathcal{C}}$. Then the induced fibrations $p_{\alpha_0}\colon A\times_B^{\alpha_0} X\Right0{} A$ and
$p_{\alpha_1}\colon A\times_B^{\alpha_1} X\Right0{} A$ are fibrewise homotopy equivalent.
\end{Lem}

\begin{Def}
The fibre $F_b$ of a $\calc$-fibration $p\colon X\Right0{} B$
over a vertex $b\in B$ is
defined by the pullback diagram
$$
\xymatrix{ F_b \ar[r]^-{\5{b}} \ar[d]_{\5p} & X \ar[d]^p \\  \Delta[0] \ar[r]^{b} & B}
$$
We will say that the fibre of the fibration $p\colon X\Right0{} B$ is weakly equivalent to the
$\calc$-diagram $F$ if for every vertex $b\in B$,
the fibre $F_b$ over $b$ is weakly homotopy equivalent to $F$.
\end{Def}

\begin{Def}\label{w-eq-fibs} Two fibrations $p\colon X\Right0{} B$ and $p'\colon X'\Right0{} B$ in $\mathbf{S}^{\mathcal{C}}$ are said to be weakly homotopy equivalent if there exists a finite zig-zag
\begin{equation}\label{seq_w.eq}
 \xymatrix{ X \ar[d]^{p}\ar[r]^{\simeq} & X_1 \ar[d]^{p_1} &  X_2 \ar[d]^{p_2}\ar[l]_{\simeq}\ar[r]^{\simeq} & ... &
X_n \ar[d]^{p_n}\ar[l]_{\simeq}\ar[r]^{\simeq}  & X' \ar[d]^{p'} \\
 B\ar@{=}[r]  & B\ar@{=}[r]   &  B\ar@{=}[r]  &  \dots \ar@{=}[r]  & B\ar@{=}[r]   & B
}
\end{equation}
of weak equivalences over $B$ that
joins $p$ with $p'$ through fibrations of $\mathbf{S}^{\mathcal{C}}$.
\end{Def}

In the situation of the above definition, for a vertex $b$ of $B$, we obtain a corresponding zig-zag of weak homotopy equivalences between the respective fibres
\begin{equation}\label{seq-fibs}
 F_b \Right0{\simeq} (F_1)_b \Left0{\simeq} (F_2)_b\Right0{\simeq}
       \dots \Left0{\simeq} (F_n)_b\Right0{\simeq} F'_b
\end{equation}

\begin{Nota}\label{rem-w.eq.}
 Note that we would obtain an equivalent definition to \ref{w-eq-fibs} if we did not require the maps $p_i\colon
 X_i\Right0{}B$,
 $i=1, \dots, n$
 to be fibrations. Indeed, if $p_n$ is not a fibration, we can factor it as a trivial cofibration followed by a fibration
 $X_n\Right0{\imath} X_n'\Right0{p_n'} B$, by axioms M5 (cf.~Definition~\ref{D:Model-Category}) and then axiom
 M4 would provide liftings $X_{n-1} \Left0{\simeq} X_n' \Right0{\simeq} X'$, making the diagram
\begin{equation}\label{seq_w.eq-b}
 \xymatrix{  X_{n-1}  \ar[d]_{p_{n-1}}&  X'_n \ar[d]^{p'_n}\ar[l]_-{\simeq}\ar[r]^-{\simeq}  & X' \ar[d]^{p'} \\
 B  \ar@{=}[r]  & B\ar@{=}[r]   & B
}
\end{equation}
commutative, with $p_n'$ a fibration.
Similarly for the other maps $p_i$.
\end{Nota}

\begin{Lem}\label{fibres} If  $p\colon X\Right0{} B$ is a minimal $\calc$-fibration, the fibre $F_b$ over a vertex $b\in B$
is a minimal $\calc$-diagram. Moreover,
\begin{enumerate}[\rm (a)]
\item If $b'$ is another vertex of $B$ in the same connected component as $b$, then $F_b$ and $F_{b'}$  are isomorphic.
\item If the fibrations in the sequence of weak equivalences \eqref{seq_w.eq} are minimal, then the maps
in the sequence \eqref{seq-fibs} are isomorphisms.
\end{enumerate}
\end{Lem}
\begin{proof} $F_b$ is minimal by
Lemma~\ref{L:pullback of min.f}.

Then, (a) follows because if $b$ and  $b'$ are the two vertices of an edge of $B$,
Lemma~\ref{L:f.h.e} implies that the pullback of $p$ along the inclusion of this edge in
$B$ provides a homotopy equivalence between $F_b$ and $F_{b'}$. Now,
Corollary~\ref{C:Isomorphism between minimal fibrations} implies that the
map is indeed an isomorphism. In general there will be a sequence of edges joining
$b$ and $b'$, and we only need to apply this same argument repeatedly along these lines.

Finally, (b) follows by similar arguments.
\end{proof}

\begin{Cor}\label{C:Min.fib is FB}
Let $B$ be a connected simplicial set.
If $p\colon X\Right0{} B$ is a minimal fibration in $\mathbf{S}^{\mathcal{C}}$ over the constant diagram
$B$, then $p$ is a $\calc$-fibre bundle over $B$ with fibre $F\cong F_b$ for any vertex $b\in B$.
\end{Cor}

\begin{proof} Since $\Delta[n]$ is contractible, the result follows from Lemma~\ref{L:pullback of min.f}, Corollary~\ref{C:Isomorphism between minimal fibrations} and  Lemma~\ref{L:f.h.e}.
\end{proof}

\begin{Def}
 A minimal cofibrant-fibrant replacement for $F$ is a  $\calc$-diagram $MF$, weakly equivalent to $F$,
 which is a minimal  $\calc$-diagram in the sense of Definition~\ref{D:Minimal fibration-diagrams}.
\end{Def}

\begin{Pro}\label{(C)1stPart} Let $\mathcal{C}$ be a small artinian $EI$-category.
\begin{enumerate}[\rm (a)]
\item  Each $\calc$-diagram $F$ admits a minimal cofibrant-fibrant replacement.
If $F$ is a fibrant $\calc$-diagram, there is a choice of a minimal cofibrant-fibrant replacement $MF$ for which there is a weak equivalence
\begin{equation}\label{Min-repla}
  \kappa\colon MF \Right2{\simeq_w} F\,.
\end{equation}

\item If $B$ is a connected simplicial set,
any fibration of $\calc$-diagrams $p\colon X\Right0{} B$ over the
constant diagram $B$
with fibre weakly homotopy equivalent to a $\calc$-diagram $F$
is weakly homotopy equivalent
to a minimal fibration $\4p\colon MX \Right0{} B$ with fibres isomorphic to $MF$,
a minimal cofibrant-fibrant replacement of $F$.
\end{enumerate}
\end{Pro}
\begin{proof}\textbf{(b) }
We build a diagram of fibrations
\begin{equation}\label{constminimalfib}
\xymatrix{ X\ar[d]_p     & QX\ar[d]_{p'}
                      \ar[l]_(0.50){\simeq}   & MX\ar[d]^{\widehat{p}}\ar[l]_{\simeq} \\
             B    \ar@{=}[r]     & B \ar@{=}[r]  & B                         }
\end{equation}
by first applying Quillen's small objects argument to the $\calc$-diagram $X$, thus obtaining
a trivial fibration $QX\Right0{}X$ with $QX$ a free $\calc$-diagram (see Remark \ref{R:Quillen-SOA}).
The map $p'\colon QX\Right0{}B$, making the square in the left of the above diagram
commutative, is then defined as the composition of two fibrations and it is therefore a fibration.
By Theorem \ref{T:const-min-fibrat(A)}, $p'\colon QX\Right0{}B$ has a fibrewise deformation retract
$\4p\colon MX\Right0{} B$
which is a minimal fibration.

By Lemma~\ref{fibres}(a) the fibres over the different vertices of $B$ are isomorphic.
We can set $MF= (MX)_b $, the fibre of $\4p$ over a chosen vertex $b\in B$.

By pulling back the fibrations of \eqref{constminimalfib} along the inclusion of the vertex $b$,
we get a map between the fibres of $p$ and of $\4p$
which is a weak homotopy equivalence
$X_b \Left0{} (MX)_b$. By assumption $ X_b\simeq_w F$ are weak homotopy equivalent.
$MF = (MX)_b$ is a minimal cofibrant-fibrant $\calc$-diagram, hence a minimal cofibrant-fibrant
replacement for $F$.

\noindent \textbf{(a) } This follows now from (b) with $B=*$ a point and $X=F$.
If $F$ is a fibrant $\calc$-diagram, then \eqref{constminimalfib} provides the map
$\kappa\colon MF \Right0{\simeq_w} F$.

If $F$ is not fibrant, then we take first a fibrant approximation $F\Right0{}\4F$, and then apply the above construction to $\4F$ and get $F\Right0{\simeq_w}\4F \Left0{\ \simeq_w}M\4F$.
\end{proof}

\begin{Pro}\label{(C)2ndPart} Let $\mathcal{C}$ be a small artinian $EI$-category.
\begin{enumerate}[\rm (a)]
\item  Two minimal cofibrant-fibrant replacements for a $\calc$-diagram $F$ are isomorphic.
\item  Two weakly homotopy equivalent minimal fibrations of $\calc$-diagrams over a constant
diagram $B$ are isomorphic.
\end{enumerate}
\end{Pro}
\begin{proof}
 \textbf{(b) } Assume that $p\colon M\Right0{} B$ and
$p'\colon M'\Right0{} B$ are minimal fibrations of $\calc$-diagrams.
If they are weakly
homotopy equivalent,  then there is a finite zig-zag of weak equivalences of fibrations
\begin{equation}\label{zig-zagmin}
\xymatrix{\llap{$M=\,$} X_0 \ar[r]^-{\simeq_w} \ar[d]_p &  X_1 \ar[d]  &   \ar[l]_-{\simeq_w}
                         X_2 \ar[r]^-{\simeq_w}  \ar[d]
       &  X_3 \ar[d]\ar[r]^-{\simeq_w}   &  \dots \ar[r]^-{\simeq_w}
       &  X_{n-1}   \ar[d] &    \ar[l]_-{\simeq_w}  X_n \ar[d]^{p'}  \rlap{$\,=M'$}\\
 B \ar@{=}[r]& B \ar@{=}[r] & B \ar@{=}[r] & B\ar@{=}[r]  & \dots \ar@{=}[r]  & B \ar@{=}[r]
 & B\rlap{.}
}
\end{equation}
Upon
replacing $X_i$ by   $QX_i$, $i=1,\dots, n-1$
 (see Remark~\ref{R:Quillen-SOA}) we can assume that all of the $X_i$'s are
  cofibrant $\calc$-diagrams. $M$, and $M'$ are also cofibrant since the fibrations $p$, $p'$ are minimal.

Whitehead theorem for
model categories  implies that these are indeed homotopy equivalences.
We will only sketch some small changes to the classical proof
(cf.~\cite{Goerss-Jardine, Hirschhorn})
that in our situation avoids the model category structure of
$\Cdiag\dn  B$.

We can assume that each of the maps in the top horizontal line of digram~\eqref{zig-zagmin}
is either a trivial cofibration or a trivial fibration. Otherwise, factor a weak equivalence
$X_i\Right0{} X_{i\pm1}$
 into
 a cofibration followed by a trivial fibration $X_i\Right0{}Z\Right0{} X_{i\pm1}$
over $B$. It turns out that $X_i\Right0{}Z$ is actually a trivial cofibration
and $Z\Right0{}B$ is a fibration with $Z$ a cofibrant $\calc$-diagram.

Assume that one of the maps $f\colon X_i\Right0{}X_{i\pm1}$ is
 a trivial cofibration.
In this case, we obtain  a left inverse
$g$ of $f$ over $B$ as a dotted lift in the following solid diagram
 \begin{equation*}
\xymatrix@C=3em{X_i  \ar@{=}[r]  \ar[d]_f   &   X_i  \ar[d]^{p_i} \\
    X_{i\pm1} \ar@{.>}[ru]^g  \ar[r]_{p_{i\pm1}} &  B
      }
\end{equation*}
A homotopy $h$ over $B$
 between $g\circ f$ and $\Id_{X_i}$ is provided by the dotted lift in the following solid diagram
  \begin{equation*}
\xymatrix@C=8em{\dot\Delta[1]\times X_{i\pm1} \bigcup_{\dot\Delta[1]\times {X_{i}}} \Delta[1]\times {X_{i}}
 \ar[r]^-{(f\circ g\cup \Id_{X_{i\pm1}})\bigcup f\circ\pr}
 \ar[d]_{\incl\times\Id_{X_{i\pm1}}\cup \Id_{\Delta[1]}\times f}
                                       &   X_{i\pm1}  \ar[d]^{p_{i\pm1}}   \\
   \Delta[1]\times X_{i\pm1}  \ar@{.>}[ru]^h  \ar[r]^-{p_{i\pm1}\circ \pr}   &  B
      }
\end{equation*}
where the left vertical map is a trivial cofibration by Proposition~\ref{cof}.

A similar argument shows that if $f\colon X_i\Right0{}X_{\pm1}$ is a trivial fibration,
then it is a homotopy equivalence over $B$. Since each of the maps $p_i\colon X_i\Right0{}B$
is a fibration, the composition of the fibrewise homotopy equivalences of
diagram~\eqref{zig-zagmin} is a fibrewise homotopy equivalence $M\Right0{}M'$ over $B$.
Now, the result follows by Corollary~\ref{C:Isomorphism between minimal fibrations}.

\noindent \textbf{(a) } This follows now from (b), taking $B=*$ and $M=F$.
\end{proof}

\medskip

By Corollary~\ref{C:Min.fib is FB}, and assuming that $B$ is connected,
a minimal fibration $\4p\colon MX \Right0{} B$  with
fibres isomorphic to $MF$, like the one obtained
in Proposition~\ref{(C)1stPart}(b), is indeed a
 a fibre bundle
with fibre $MF$ and structural group
$\aut_\calc(MF)$, by default.
We will describe
the homotopy type of $\aut_\calc(MF)$ in terms of the original $F$, assuming that $F$ is fibrant.
This is based on work of Dwyer and Kan on function complexes \cite{Dwyer}.
We will need to introduce now the notion of  twisted arrow category.

\begin{Def}
Let $\calc$ be a small category. The \emph{twisted arrow category} $\acalc$ of $\calc$
is the category whose objects are the morphisms $f\colon a\Right0{}b$
of $\calc$,  and morphisms from $f\colon a\Right0{}b$ to $f'\colon a'\Right0{}b'$ the pairs of
morphisms $(\alpha,\beta)$, where $\alpha\in\Mor_\calc(a', a)$,  $\beta\in\Mor_\calc(b,b')$,
and the diagram
$$
 \xymatrix{ a \ar[r]^f  &   b\ar[d]^\beta  \\
                 a'\ar[u]^\alpha \ar[r]^{f'} &  b'
 }
$$
commutes in $\calc$.
\end{Def}

Following \cite{Dwyer}, if  $X$, $Y$ are $\calc$-diagrams, we define a
diagram of function complexes, indexed by the twisted  arrow category $\acalc$,
$$\mapa (X,Y)\colon \acalc\Right2{} \sSets$$
that maps an object $f\colon a\Right0{}b$ to $\map(X(a), Y(b))$ and a morphism $(\alpha, \beta)$ from
$f$ to $f'$ to the map
$\alpha^\sharp\beta_\sharp\colon \map(X(a), Y(b))\Right0{}\map(X(a'), Y(b'))$
induced by right and left composition.

 Naturality of these function complex diagrams is easily checked. A map of $\calc$-diagrams
 $\theta\colon X'\Right0{}X$ induce a map of $\acalc$-diagrams
 $\theta^*\colon \mapa(X,Y)\Right0{}\mapa(X',Y)$,  which for each $\varphi\in \map(X(a), Y(b))$ is defined over every object $f\colon a\Right0{} b$ of $\acalc$,
 by the precomposition
 $\theta^*_f(\varphi)=\varphi\circ \theta(a)$.
  Similarly, a map  of $\calc$-diagrams
 $\theta\colon Y\Right0{}Y'$ induce a map of $\acalc$-diagrams
 $\theta_*\colon \mapa(X,Y)\Right0{}\mapa(X,Y')$, by postcomposition.
 It follows that the function complex diagram
 is a bifunctor
\begin{equation}\label{mapafunctor}
 \mapa(-,-)\colon (\Cdiag)\op \times \Cdiag \Right2{} \aCdiag\,.
\end{equation}

One clearly has an isomorphism
$\map_\calc(X,Y) \cong \varprojlim_{\acalc}\mapa (X,Y)$
and a natural map
\begin{equation}\label{lim-holim}
 \varprojlim_{\acalc}\mapa (X,Y) \Right2{\simeq} \holim_{\acalc}\mapa (X,Y)
\end{equation}
which according to \cite[Theorem~3.3]{Dwyer} is a weak homotopy equivalence.
Recall that the inverse limit
$\varprojlim_{\acalc}\mapa (X,Y) $ can be defined as the function complex
$\map_{\acalc}(*,\mapa (X,Y))$, where $*$ stands for a constant functor with value a single point.
Likewise,
\begin{equation}\label{h.inv.lim}
 \holim_{\acalc}\mapa (X,Y) = \map_{\acalc}\bigl( E(\acalc)\,,\,\mapa (X,Y)\bigr)\,,
\end{equation}
where
$E(\acalc)$ is the cofibrant replacement of $*$ in the model category $\aCdiag$,
which for each object $f$ of $\acalc$,
$E(\acalc)(f)$ is defined as the nerve of the overcategory
$\acalc\dn  f$.  $E(\acalc)$ is indeed a free
$\acalc$-diagram.
The map in
\eqref{lim-holim} is then
induced by the natural projection $E(\acalc)\Right0{}*$.

\begin{Lem}\label{fibrantmapa} Let $\calc$ be a small category and $X$ and $Y$ $\calc$-diagrams. Assume that $Y$
is fibrant, then
\begin{enumerate}[\rm (a)]
\item $\mapa(X,Y)$ is a fibrant $\acalc$-diagram, and
\item  $\holim_{\acalc}\mapa (X,Y)$ is a Kan complex.
\end{enumerate}
\end{Lem}
\begin{proof}
Since $Y$ is fibrant, for each $c\in \calc$, $Y(c)$ is a Kan complex. Then, for
each object $f\colon a\Right0{}b$ in the arrow category, $\mapa(X,Y)(f)= \map(X(a),Y(b))$
is also a Kan complex
(cf.~\cite[I.6.9]{May}),  so  $\mapa(X,Y)$ is a fibrant $\acalc$-diagram.
Using the description of the homotopy inverse limit in equation \eqref{h.inv.lim}, (b) follows
essentially from axiom M\ref{axM7} in Definition \ref{D:Model-Category}
(see also \cite[II.3.2]{Goerss-Jardine},
\cite[9.3.1]{Hirschhorn}), since
$E(\acalc)$ is cofibrant and $\mapa(X,Y)$ is fibrant.
\end{proof}

For each object $f\colon a\Right0{} b$ of $\acalc$, $E(\acalc)(f)$ has a canonical base point, the vertex
$v_f$ corresponding to the terminal object of the overcategory $\acalc\dn f$. Evaluating at
the vertices $v_{\Id_a}$, $a\in \Ob(\calc)$, we obtain a map
\begin{equation}
\holim_{\acalc}\mapa (X,Y)  \Right2{\ev_\Id} \prod_{a\in\acalc} \map(X(a), Y(a))
\end{equation}
such that the composition
\begin{equation}\map_\calc(X,Y) \cong
\varprojlim_{\acalc}\mapa (X,Y) \Right0{}
\holim_{\acalc}\mapa (X,Y) \Right0{}
 \prod_{a\in\acalc} \map(X(a), Y(a))
\end{equation}
 maps a natural transformation $\eta\colon X\Right0{} Y$ to
 the sequence
 $\con{\eta_a\colon X(a)\Right0{} Y(a)}_{a\in \Ob(\acalc)}$
of the natural maps defined for the different objects of $\calc$.
This will help to identify the connected components of
$\holim_{\acalc}\mapa (X,Y) $ that correspond to
invertible maps in $\map_\calc(X,Y)$, as stated in the following definition.

\begin{Def}\label{defhaut}
 Let $X$ be a fibrant $\calc$-diagram. We define the space   $\haut_\calc(X)$
 of self homotopy equivalences  of $X$ as the subspace
$ \bigl[\,\holim_{\acalc}\mapa (X,X) \,\bigr]_{\textup{we}}$
 of $\holim_{\acalc}\mapa (X,X) $ consisting of the connected components of the
 vertices  $\omega$  of $\holim_{\acalc}\mapa (X,Y)$ such that the evaluations
 $\ev_\Id(\omega) = \con{\omega_a\colon X(a)\Right0{}Y(a)}_{\a\in \Ob(\calc)}$
 are weak equivalences. In case $X$ is not fibrant, we define
 $\haut_\calc(X)\defeq\haut_\calc(\4X)$, where $\4X$ is a fibrant replacement of $X$.
\end{Def}

\begin{Pro}\label{haut}
If $X$ is a fibrant $\calc$-diagram, then $\haut_\calc(X)$ is a loop space with classifying space
$B\haut_\calc(X) \simeq  \3W\aut_\calc(MX) $,
where $MX$ is a minimal cofibrant-fibrant replacement for $X$.
\end{Pro}
\begin{proof}
Notice first that according to
Proposition~\ref{(C)2ndPart}(a), the conclusion of the proposition does not depend on the choice of $MX$.
Furthermore, we can fix a model for $MX$ together with weak homotopy equivalence
$\kappa\colon MX\Right0{\simeq_w}X$ (cf.~Proposition~\ref{(C)1stPart}(a)).

We need to prove that the space of loops $\Omega(\3W\aut_\calc(MX))$ is homotopy equivalent to
$\haut_\calc(X)$, so it will be enough to show a homotopy equivalence $\aut_\calc(MX)\simeq \haut_\calc(X)$.

Observe that by \cite[Theorem~3.3]{Dwyer}, there is a sequence of homotopy equivalences
\begin{equation}\label{h.eq}\map_\calc(MX,MX)
\Right2{\cong} \varprojlim_{\acalc}\mapa (MX,MX) \Right2{\simeq}
               \holim_{\acalc}\mapa (MX,MX)\,.
\end{equation}

Since $\kappa\colon MX\Right0{} X$ is a weak equivalence  of $\calc$-diagrams,
for each object $c$ of $\calc$, the map $\kappa_c\colon (MX)(c) \Right0{} X(c)$ is a weak
 equivalence of simplicial sets. By naturality of $\mapa$ \eqref{mapafunctor}, $\kappa$ induces maps
 of $\acalc$-diagrams
\begin{equation}\label{h.eq.1}
\mapa (X,X)\Right2{\kappa^*}   \mapa (MX,X)\Left2{\kappa_*} \mapa (MX,MX)\,.
\end{equation}
By assumption $X$ is fibrant and so is $MX$, hence by Lemma~\ref{fibrantmapa}, all three
$\acalc$-diagrams
in the above equation \eqref{h.eq.1} are fibrant. Furthermore the maps $\kappa^*$ and
$\kappa_*$ are weak equivalences in $\aCdiag$. In fact,
for each object $f\colon a\Right0{}b$ of $\acalc$, the natural map
 $$
     \kappa^*_f\defeq \kappa_a^\sharp\colon \map\bigl(X(a), X(b)\bigr) \Right0{} \map\bigl((MX)(a),X(b)\bigr)\,,
 $$
is given by precomposition with the map $\kappa_a\colon (MX)(a)\Right0{}X(a)$ which is a weak
equivalence of simplicial sets and $X(b)$ is a Kan complex; therefore we obtain that $\kappa^*_f= \kappa_a^\sharp$ is a weak
equivalence (cf.~\cite[I.6.9]{May}).
A similar argument applies to $\kappa_*$.

 Now, Bousfield-Kan homotopy lemma \cite[XI.5.6]{Bousfield-Kan} applies to show that the induced maps
\begin{equation}\label{h.eq.2}
 \holim_{\acalc}\mapa (X,X)\Right8{ \holim_{\acalc}\kappa^*} \holim_{\acalc}\mapa (MX,X)
 \Left8{ \holim_{\acalc}\kappa_*} \holim_{\acalc}\mapa (MX,MX)
\end{equation}
are weak homotopy equivalences. In other words, this follows from \cite[9.3.3(1)]{Hirschhorn},
by our description of homotopy
inverse limit in Equation \eqref{h.inv.lim} and because the diagram
$E(\acalc)$ is cofibrant in the simplicial model category
of $\acalc$-diagrams.

 The weak equivalences of equations \eqref{h.eq} and  \eqref{h.eq.2}
 combine to provide the following
 zig-zag diagram
 \begin{equation}\label{h.eq.3}\map_\calc(MX,MX) \Right2{}
 \holim_{\acalc}\mapa (MX,X) \Left2{}
 \holim_{\acalc}\mapa (X,X)\,.
 \end{equation}
 Since the above maps are weak equivalences between Kan complexes,
they are also homotopy equivalences (cf.~\cite[II.1.10]{Goerss-Jardine}) and we
 we have already obtained a
 homotopy equivalence $\map_\calc(MX,MX) \simeq  \holim_{\acalc}\mapa (X,X)$.
It only remains to show that the connected components of $\map_\calc(MX,MX) $ that consist of isomorphisms
correspond to those of $\haut_\calc(X)$ in $\holim_{\acalc}\mapa (X,X)$, as defined in \ref{defhaut}, through
the maps in \eqref{h.eq.3}.

Choose a vertex $\eta\in \aut_\calc(MX) \subseteq \map_\calc(MX,MX)$, that is, a natural isomorphism
$\eta\colon MX\Right0{} MX$ of $\calc$-diagrams. The image of $\eta$ along the left map in
 diagram \eqref{h.eq.3}, will be  a vertex $\5\eta\in  \holim_{\acalc}\mapa (MX,X)$
such that $\ev_\Id(\5\eta)= \con{\5\eta_a\colon (MX)(a) \Right0{}X(a)}_{\a\in \Ob(\calc)}$
where $\5\eta_a$ is the composition
$$
MX(a)\Right0{\eta_a}MX(a)\Right0{f_a}X(a),
$$
which is clearly a homotopy equivalence. Now, let $\omega$ be a vertex of $ \holim_{\acalc}\mapa (X,X)$
mapping to the same connected component as $\5\eta$ along the right map of diagram \eqref{h.eq.3}.
It follows that if $\ev_\Id(\sigma) = \con{\omega_a\colon X(a)\Right0{}X(a)}_{\a\in \Ob(\calc)}$,
there is a diagram
\begin{equation}\label{equatingcomponents}
 \xymatrix{ (MX)(a)\ar[d]^{f_a} \ar[r]^{\eta_a}& (MX)(a) \ar[d]^{f_a} \\
                   X(a) \ar[r]^{\omega_a}& X(a) \rlap{ ,}
}
\end{equation}
for each object $a\in \calc$, that commute up to homotopy.
It follows that each $\omega_a$ is a weak homotopy equivalence,
and $\omega$ is a vertex of $\haut_\calc(X)$ by Definition~\ref{defhaut}.

Conversely, choose
 a 0-simplex $\omega$ of $\haut_\calc(X)\subseteq\holim_{\acalc}\mapa (X,X)$ and write
$\ev_\Id(\omega) = \con{\omega_a\colon X(a)\Right0{}X(a)}_{\a\in \Ob(\calc)}$, where
each $\omega_a$ is a weak homotopy equivalence.
Observe that $\omega$ determines a connected component
in  $\holim_{\acalc}\mapa (X,X)$.
Choose a vertex $\eta$ in the corresponding connected component of
$\map_\calc(MX,MX)$, via the homotopy equivalences \eqref{h.eq.3}.
We obtain a homotopy commutative diagram like~\eqref{equatingcomponents}, above,
and conclude that $\eta_a\colon (MX)(a)\Right0{} (MX)(a)$ is a weak homotooy
equivalence for each object $a$ of $\calc$. Hence $\eta\colon MX\Right0{} MX$
is a weak equivalence of $\calc$-diagrams. Since $MX$ is cofibrant-fibrant,
by Whitehead's Theorem $\eta_a$ is a homotopy equivalence. But $MX$ is also a
minimal $\calc$-diagram, hence $\eta$ is
indeed an isomorphism by Corollary~\ref{C:Isomorphism between minimal fibrations}.
Thus $\eta$ is really a vertex of $\aut_\calc(MX)$.
\end{proof}

\begin{proof}[Proof of  Theorem~\ref{T:Clas-C-Fib-1(C)}]
We are assuming that $\mathcal{C}$ is a small artinian $EI$-category,
$B$ a connected simplicial set, and $F$ an arbitrary $\mathcal{C}$-diagram. Let $MF$ be a
minimal cofibrant-fibrant replacement for $F$.

Since $MF$ is minimal, a $\calc$-fibre bundle over $B$ with fibre $MF$,
$\xi\colon E\Right0{}B$, is indeed a
minimal fibration. For each object $c$ of $\calc$,  $(MF)(c)$ is a Kan complex and
$\xi(c)\colon E(c)\Right0{}B$ is a fibre bundle with fibre $(MF)(c)$, hence
a fibration
(cf.~\cite{Barrat}). Thus, $\xi\colon E\Right0{}B$ is a fibration of $\calc$-diagrams.

That $\xi$ is minimal follows from Proposition~\ref{P:Other.Charact.Min.fibrat}.
If $\xi|_D\colon D\Right0{} B$ is a strong fibrewise deformation retract of $\xi$, then for each
simplex $\sigma$ of $B$, the pullback along the inclusion of the simplex
$\sigma\colon \Delta[n]\Right0{}B$ produces a retraction
$$
\xymatrix{ D_\sigma \ar[r]^-{\incl} \ar[d]_{\4{\xi|_{D}}}
             & \Delta[n]\times MF \ar[r]^-{r}\ar[d]^{\pr}
             & D_\sigma\ar[d]^{\4{\xi|_{D}}}  \\
           \Delta[n]\ar@{=}[r] &  \Delta[n]\ar@{=}[r]  &  \Delta[n].
             }
$$
Moreover, since $MF$ is minimal, $\pr\colon \Delta[n]\times MF\Right0{} \Delta[n]$ is a minimal fibration by
Lemma~\ref{L:pullback of min.f}.
Then, Proposition~\ref{P:Other.Charact.Min.fibrat} implies that $\incl\colon D_\sigma\Right0{}
\Delta[n]\times MF$ is an isomorphism, and therefore $\incl\colon D\Right0{} E$ is an isomorphism.
Now, Proposition~\ref{P:Other.Charact.Min.fibrat} implies that $\xi\colon E\Right0{}B$ is minimal.

 Now we can define the correspondence
\begin{equation*}
 \left\{  \parbox{4cm}{\flushleft Equivalence classes
of $\calc$-fibre bundles over $B$ with fibre $MF$}      \right\}
    \RIGHT6{J}{\cong}
 \left\{  \parbox{6.5cm}{\flushleft Weak homotopy classes
of fibrations of $\calc$-diagrams over the constant diagram
$B$ and fibre weakly homotopy equivalent to $F$}      \right\}
\end{equation*}
by assigning to the class of $\calc$-fibre bundle $\xi\colon E\Right0{}B$
the weak homotopy class
of the same map
as a fibration of $\calc$-diagrams with fibre $MF\simeq_w F$.

Proposition~\ref{(C)1stPart}(b) together with Corollary~\ref{C:Min.fib is FB}
imply that $J$ is surjective while
Proposition~\ref{(C)2ndPart}(b) implies that it is injective.
Then, the result follows from Theorem~\ref{T:Classif-C-fibre(B)} and Proposition~\ref{haut}.
\end{proof}

\appendix

\section{Preordered sets}

Let $A$ be a set and `$\preceq$' a preorder relation over $A$ (i.e., $\preceq$ is reflexive and transitive).
For this set we need to find a subset $A'$ of $A$ satisfying the conditions:
\begin{enumerate}[\bf R1:]
\item For all $w\in A$ there exists $x\in A'$ with $x\preceq w$.
\item $A'$ is minimal among subsets satisfying R1.
\end{enumerate}

There is a simple example which shows
that $A'$ does not always exist: consider an infinite sequence of sets and maps
\begin{equation*}
\xymatrix{\dots \ar[r]  & A_n\ar[r]^{f_n}  & \dots \ar[r] & A_1\ar[r]^{f_1}  & A_0 }
\end{equation*}
with nonempty inverse limit. Let us define the relation `$\preceq$' over the disjoint union
$A=\coprod_{i\geq 0}A_i$ as follows: given $a\in A_j$ and $a'\in A_k$ we say
that $a\preceq a'$ if there exists a map $f\colon A_j\Right0{} A_k$ such that $f(a)=a'$.
In this case there is no subset $A'\subseteq A$ satisfying both conditions R1 and R2.
In fact, any element in the inverse limit provides an infinite
sequence  $a_{0}\succeq a_{1}\succeq \dots \succeq a_{i}\succeq \dots $,  in $A$, with $a_i\in A_i$,
 which does not stabilize.

We can reformulate this problem by defining an order relation over a convenient quotient of $A$,
that is, over the preordered set $(A,\preceq)$ we define the following relation:
given $x,w\in A$ we say that $x\sim w $ if $x\preceq w$ and $w\preceq x$.
Note that `$\sim$' is an equivalence relation over $A$. Now, we define an partial order
relation over  $A/{\sim}$.  Given classes $[x],[w]\in A/{\sim}$ we set
$$
     [x]\leq[w]\quad \text{if} \quad x\preceq w\,.
$$
With this new relation the couple $(A/{\sim},\leq)$ is a partially ordered set, and then the existence of a subset
$A'$ satisfying the above required conditions depends on the existence of minimal elements in every maximal
chain of $A/{\sim}$.

Let $\calc$ be a small category and let $X\colon \calc\Right0{}\Sets$ be a functor.
We will say that $X$ is a diagram of sets and equivalence relations if
 for every object $c$ of $\calc$ there is an equivalence relation
$\sim_c$ defined over $X_c$, and is natural in the sense that
 for every
$f\colon c\Right0{} d$ in $\calc$, if $x\sim_c x'$ then  $X_f(x)\sim_d X_f(x')$.

Under these circumstances we can define a new induced relation $\preceq$ over the
disjoint union $A=\coprod_{c\in\Ob(\calc)}X_c$ as follows:
given $x\in X_c$ and $w\in X_d$ we say that $x\preceq w$ if there exists a morphism
$f\colon c\Right0{} d$ in $\calc$ such that $X_f(x)\sim_d w$.
Now, the pair $(A,\preceq)$ is a preordered set.

\begin{Pro}\label{E:Existen-of-A'}
Let $\calc$ be an artinian $EI$-category. If $X\colon \calc\Right0{}\Sets$ is a
diagram of sets and equivalence relations, then the set $A= \coprod_{c\in\Ob(\calc)}X_c$
equipped with the induced preorder relation $\preceq$ defined above contains a subset
$A'$ satisfying the conditions R1 and R2.
\end{Pro}

\begin{proof}
Define an equivalence relation $\sim$ on $A$ by setting
$x\sim x'$ if both $x\preceq x'$ and  $x'\preceq x$. Write $[x]$ for the class of $x$ in $A/{\sim}$.
Now, define a partial order relation on $A/{\sim}$ by setting
$[x]\leq [y]$ if $x\preceq y$. Consider the set of all maximal chains of $A/{\sim}$:
$$
   \mathcal{F}=\Conj{C\subseteq A/{\sim}}{C \ \text{is a maximal chain}}\,.
$$

Take an arbitrary maximal chain $C = \{[x_i]\}_{i\in I}$ in $A/{\sim}$ with $x_i\in X_{c_i}$,
$c_i\in \Ob(\calc)$, for each $i\in I$. The classes $[c_i]$ of these objects $c_i$, $i\in I$,
form a chain in $\calc/{\sim}$, so it stabilizes
since $\calc$ is artinian (see Definition~\ref{D:wee-descendant}).
In fact, given two of these objects $c_i$ and $c_j$,
we have elements $x_i\in X_{c_i}$ and $x_j\in X_{c_j}$ with $[x_i], [x_j]\in C$, thus either
$[x_i] \leq  [x_j]$ in which case there is a morphism $f\colon c_i \Right0{} c_j$ in $\calc$
and $[c_i]\leq [c_j]$, or $[x_j] \leq  [x_i]$ in which case $[c_j]\leq [c_i]$.

There is a representative  $c_0\in \Ob{\calc}$ of the minimal element $[c_0]$ in the chain of
classes of objects such that there is an element $x_0\in X_{c_0}$ with $[x_0]\in C$.
It turns out that  $[x_0]$ is minimal in $C$.
Assume $[z]\in C$ and $[z] \leq [x_0]$, then $z\in X_{d}$ for some object $d$ of $\calc$
and there is a  morphism $f\colon d\Right0{} c_0$ in $\calc$ such that $X(f)(z)\simeq_{c_0} x_0$. Then,
$[d]$ belongs to the chain of objects and $[d]\leq [c_0]$, hence $[d]= [c_0]$ since $[c_0]$ is minimal.
That is, we also have $[c_0]\leq [d]$, so there is a morphism $g\colon c_0\Right0{} d$.
Since $\calc$ is an $EI$-category, the  composition $\varphi=g\circ f$ is an automorphism of $d$.
Upon replacing $g$ with $\varphi^{-1}\circ g$ we can assume that $g\circ f=\Id_d$.
It follows that $z= X(g)(X(f)(z))\simeq_{d} X(g)(x_0)$, so $x_0\preceq z$ and $[x_0]\leq [z]$.
 Hence $[z] = [x_0]$.

Consider now the set $M$ of minimal elements of the chains in $\calf$,  that is
$$
    M= \Conj{[x]\in A/{\sim}}{\text{there is $C\in \calf$ such that $[x]$ in minimal in $C$}}
$$
By the Axiom of Choice, there is a subset $A'$ that contains one and only one
representative in each of the classes that form the set $M$.

 We will show that $M$ satisfies conditions R1 and R2.
Fix $w\in A$, so that $[w]$ belongs to a maximal chain $C\subseteq A/{\sim}$. Then, there is $x\in A'$
such that $[x]$ is minimal in $C$. It follows that $[x]\leq [w]$, thus $x\preceq w$.  This proves R1.

Assume that there is another subset $B$, $B\subseteq A'\subseteq A$ satisfying R1. Fix an element
$w\in A'$.  There must be an element $x\in B$, $x\preceq w$,  since $B$ satisfies R1.
Hence $[x]\leq [w]$. By definition of $A'$, there is a maximal chain
$C$ in $A/{\sim}$ such that $[w]$ is minimal in $C$. Then $D=\{[x]\} \cup C$ is also a chain and it contains $C$,
hence it must be $C$, so $[x]\in C$. Therefore $[x]=[w]$, but $x\in B\subset A'$, so both $x$ and $w$
are in $A'$, and actually $x=w$. This proves that $A'\subset B$ and hence that $A'$ satisfies R2.
\end{proof}

\begin{Nota}\label{E:Existen-of-refined-A'}
 Notice that the choice of the subset $A'$ in Proposition~\ref{E:Existen-of-A'} could be
 refined by choosing elements satisfying further properties. Let $P\colon A\to \{0,1\}$ be
 a function defined on $A$ and, using the same notation as in the proof of the proposition, set
\begin{equation*}
\4{[x]}=
 \begin{cases}
 [x] \,, & \text{if $P(y)=1$ for all $y\in [x]$}, \\
 \conj{y\in [x]}{P(y)=0 } \,, & \text{otherwise}.
\end{cases}
\end{equation*}
for each class $[x]$  of  $A/{\sim}$ which is minimal in a chain in $(A/{\sim}, \leq)$.
Then, by choosing an element in each of the sets $\4{[x]}$ we obtain a refined
subset $A'$ that satisfies R1, R2, and for each  $x\in A'$,  $P(x)=0$ if and only if there is
$y\in [x]$ with $P(y)=0$.
\end{Nota}


\begin{thebibliography}{AOV}

\bibitem{Barrat} M.G. Barratt, V.K. Gugenheim and J.C. Moore, On Semisimplicial fibre bundles, \emph{Amer. J. Math.} 81 (1959), 639--657.
\bibitem{Blomgren} M. Blomgren and W. Chach\'olski, On the classification of fibrations, \emph{Trans. Amer. Math. Soc.} 367 (2015), 519--557.
\bibitem{Bousfield} A.K. Bousfield, Homotopical localizations of spaces, \emph{Amer. J. Math.} 119 (1997), no. 6, 1321--1354.
\bibitem{Bousfield-Kan} A.K. Bousfield and D. Kan, Homotopy limits, completions and localizations, Lecture Notes in Math., 304, Springer (1972)
\bibitem{Curtis} E. Curtis, Simplicial Homotopy Theory, \textbf{6}, \emph{Adv. Math.} (1971).
\bibitem{Dwyer} W. Dwyer and D. Kan, Function complexes for diagrams of simplicial sets, \emph{Nederl. Akad. Wetensch. Indag. Math}. 86, (1983), 139--147.
\bibitem{Dwyer2} W. Dwyer, D. Kan and J. Smith, Towers of fibrations and homotopical wreath products, \emph{J. Pure Appl. Alg}. 56 (1989), 9--28.
\bibitem{Farjoun} E. Farjoun and J. Smith, Homotopy localization nearly preserves fibrations, \emph{Topology} 34 (1995), no. 2, 359--375.
\bibitem{Fritsch} R. Fritsch and R. Piccinini, Cellular structures in Topology, 19, Cambridge University Press (1990).
\bibitem{Goerss-Jardine} P. Goerss, J. Jardine, Simplicial Homotopy Theory, Birkh\"{a}user Verlag (1999).
\bibitem{Gugenheim} V. Gugenheim, On supercomplexes, \emph{Trans. Amer. Math. Soc}. 85 (1957), 35--51.
\bibitem{Hirschhorn} P. Hirschhorn, Model categories and their localization, Mathematical Surveys and Monographs 99, AMS (2003).
\bibitem{Hovey} M. Hovey, Model categories, Mathematical Surveys and Monographs 63, AMS (1999).
\bibitem{Luck} W. L\"{u}ck, Transformation groups and algebraic K-theory, Lecture Notes in Math., 1408, Springer-Verlag (1980).
\bibitem{May} J.P. May, Simplicial objects in Algebraic Topology, Van Nostrand, Princeton (1967).
\bibitem{Quillen} D. Quillen, Homotopical Algebra, Lecture Notes in Math., 43. Springer-Verlag, Berlin-New York (1967).
\end{thebibliography}
\end{document}